\documentclass[11pt,letterpaper]{article}

\def\compilemode{arxiv}%
\makeatletter
\newcommand{\onlymode}[2]{%
	\def\@tempa{#1}%
	\ifx\compilemode\@tempa
		#2%
	\fi
}
\makeatother
\newcommand{\onlyarxiv}[1]{\onlymode{arxiv}{#1}}
\newcommand{\onlysoda}[1]{\onlymode{soda}{#1}\onlymode{work}{#1}}
\newcommand{\onlywork}[1]{\onlymode{work}{#1}}

\usepackage[utf8]{inputenc}

\usepackage[top=2cm,bottom=2cm,left=2cm,right=2cm]{geometry}
\onlywork{\geometry{right=4cm}}
\onlywork{\setlength{\marginparwidth}{4cm}}

\usepackage[T1]{fontenc}
\usepackage{lmodern}
\usepackage{libertine}
\usepackage[colorinlistoftodos,bordercolor=orange,backgroundcolor=orange!20,linecolor=orange,textsize=normalsize]{todonotes}

\usepackage{amsmath}
\usepackage{mathtools}
\usepackage{amssymb}
\usepackage{amsthm}
\usepackage{bbm}
\usepackage{hyperref}
\usepackage{accents}
\usepackage{booktabs}
\usepackage{fixmath}
\usepackage[noadjust,compress]{cite}
\usepackage[capitalise,noabbrev]{cleveref}
\usepackage{enumitem}
\setlist{noitemsep}
\usepackage{comment}

\usepackage{xspace}
\usepackage{tikz}
\usepackage{amsthm}

\usetikzlibrary{fit}
\usetikzlibrary{arrows}
\usetikzlibrary{patterns}
\usetikzlibrary{calc}
\usetikzlibrary{shapes}
\usetikzlibrary{positioning}
\usetikzlibrary{math}
\usetikzlibrary{shapes.symbols}
\usetikzlibrary{decorations.pathreplacing,calligraphy}
\usepackage[scr=boondox,scrscaled=1.05]{mathalfa}
\usepackage[ruled,vlined]{algorithm2e}

\newcommand{\MWIS}{\textsc{MWIS}\xspace}

\renewcommand{\geq}{\geqslant}
\renewcommand{\leq}{\leqslant}
\newcommand{\wei}{\mathsf{w}}
\newcommand{\Oh}{\mathcal{O}}

\newcommand{\cH}{\mathcal{H}}

    \newcommand{\mH}{\mathsf{H}} 
    
\newcommand{\cX}{\mathcal{X}}

\renewcommand{\le}{\leq}
\renewcommand{\ge}{\geq}

\newcommand{\NP}{\textsf{NP}\xspace}

\newcommand{\Lpgt}{\ensuremath{\Xi_{\geq t}}}

\newcommand{\cXi}{\ensuremath{\chi}}
\renewcommand{\epsilon}{\varepsilon}

\usepackage{thmtools}
\usepackage{thm-restate}

\newcommand{\say}[1]{``#1''}

\newenvironment{proofofclaim}{\noindent\textit{Proof of Claim:}}{\hfill$\lrcorner$\medskip}

\newcommand{\N}{{\mathbb N}}

\newcommand\tw{\textsf{tw}}

\usepackage{todonotes}

\newtheorem{theorem}{Theorem}[section]
\newtheorem{lemma}[theorem]{Lemma}
\newtheorem{observation}[theorem]{Observation}
\newtheorem{proposition}[theorem]{Proposition}
\declaretheorem[name=Claim]{claim}
\crefname{claim}{Claim}{Claims}
\Crefname{Claim}{Claim}{Claims}
\newtheorem{corollary}[theorem]{Corollary}
\newtheorem{conjecture}[theorem]{Conjecture}
\theoremstyle{definition}
\newtheorem{definition}[theorem]{Definition}

\newcommand{\ThetaZero}{\Theta_0}

\newcommand{\sep}{sep}

\newcommand{\BIG}{\mathsf{BIG}}

\newcommand{\EP}{Erd\H{o}s--P\'osa\xspace}

\newcommand{\dlttwo}{\ensuremath{n_{\leq 2}}}
\newcommand{\dthree}{\ensuremath{n_3}}

\newcommand{\pack}{\mathsf{pack}}

\newcommand{\sS}{\mathcal{S}}

\newcommand{\fsim}{f_{\mathrm{Sim}}}
\newcommand{\lsim}{\lambda_{\mathrm{Sim}}}
\newcommand{\ftheta}{f_{\Theta}}

\usepackage{framed}
\usepackage{tabularx}

\newlength{\RoundedBoxWidth}
\newsavebox{\GrayRoundedBox}
\newenvironment{GrayBox}[1]%
   {\setlength{\RoundedBoxWidth}{.93\columnwidth}
    \def\boxheading{#1}
    \begin{lrbox}{\GrayRoundedBox}
       \begin{minipage}{\RoundedBoxWidth}}%
   {   \end{minipage}
    \end{lrbox}
    \begin{center}
    \begin{tikzpicture}%
       \node(Text)[draw=black!20,fill=white,rounded corners,inner sep=2ex,text width=\RoundedBoxWidth]
             {\usebox{\GrayRoundedBox}};
        \coordinate(x) at (current bounding box.north west);
        \node [draw=white,rectangle,inner sep=3pt,anchor=north west,fill=white]
        at ($(x)+(6pt,.75em)$) {\boxheading};
    \end{tikzpicture}
    \end{center}}

\newenvironment{defproblemx}[1]{\noindent\ignorespaces%
                                \FrameSep=6pt%
                                \parindent=0pt%
                \begin{GrayBox}{#1}%
                \begin{tabular*}{\columnwidth}{!{\extracolsep{\fill}}@{\hspace{.1em}} >{\itshape} p{1.5cm} p{0.86\columnwidth} @{}}%
            }{
                \end{tabular*}%
                \end{GrayBox}%
                \ignorespacesafterend
            }

\title{Induced Erd\H{o}s--P\'osa property\\for long holes, long thetas, and beyond}

\onlyarxiv{%
\author{
	Jadwiga Czyżewska\thanks{University of Warsaw, Poland (\textsf{j.czyzewska@mimuw.edu.pl}).	
		Supported by Polish National Science Centre SONATA BIS-12 grant number 2022/46/E/ST6/00143.}
	\and Tomáš Masařík\thanks{University of Warsaw, Poland (\textsf{masarik@mimuw.edu.pl}).
			Supported by the Polish National Science Centre SONATA-17 grant number 2021/43/D/ST6/03312.}
	\and Marcin Pilipczuk\thanks{University of Warsaw, Poland (\textsf{m.pilipczuk@mimuw.edu.pl}).
			Supported by Polish National Science Centre SONATA BIS-12 grant number 2022/46/E/ST6/00143.}
	\and Amadeus Reinald\thanks{University of Warsaw, Poland (\textsf{reinald@mimuw.edu.pl}).
			Supported by Polish National Science Centre SONATA BIS-12 grant number 2022/46/E/ST6/00143.}
	\and Paweł Rzążewski\thanks{Warsaw University of Technology, Poland (\textsf{pawel.rzazewski@pw.edu.pl}). Supported by the National Science Centre grant 2024/54/E/ST6/00094.}}}
\onlysoda{\author{Anonymous}}

\begin{document}
	\begin{titlepage}
	\date{}
	\maketitle
	
\begin{abstract}
    The induced \EP property in graphs relates the maximum packing of pairwise anti-adjacent copies of an object with the minimal number of neighborhoods required to hit all copies.
	In this paper, the objects we consider are long cycles and long thetas, both as induced minors.
	Let $C_{t}$ denote the cycle with $t$ vertices and let $\Theta_{t}$ be the graph consisting of three internally disjoint and anti-adjacent paths, each with $t$ internal vertices, connecting the same pair of distinct vertices.

  We show that for every fixed $t$, both $C_{t}$ and $\Theta_{t}$ have the induced Erd\H{o}s--P\'osa property with respect to the induced minor relation.
	More precisely, for every integer $k$ and a graph $G$, one of the two outcomes occurs:
	\begin{itemize}
    \item $G$ contains $k$ pairwise vertex-disjoint and anti-adjacent copies of $C_{t}$ (resp., $\Theta_{t}$) as induced minors, or
	\item there is some $X \subseteq V(G)$ of size $\mathcal{O}(tk \log k)$ such that the set $N[X]$, consisting of $X$ and its neighbors, hits all $C_{t}$ (resp., all $\Theta_{t}$) induced minors in $G$.
	\end{itemize}
	This resolves in a strong form a special case of a conjecture of Ahn, Gollin, Huynh, and Kwon [SODA 2025].

	From these results we derive that graphs that exclude $k \Theta_{t}$ as an induced minor admit balanced separators consisting of the neighborhood of $\mathcal{O}(tk \log k)$ vertices.
	This in turn resolves a special case of a conjecture of Gartland and Lokshtanov and, combined with known techniques, yields a QPTAS for \textsc{Maximum Weight Independent Set} and a number of its generalizations. 
\end{abstract}

	\thispagestyle{empty}
\end{titlepage}

\tableofcontents
\thispagestyle{empty}
\newpage\setcounter{page}{1}

	\section{Introduction}
	Many classic combinatorial problems can be expressed as \emph{packing} versus \emph{hitting} certain objects in a graph.
The interplay between these two notions is a well established area of research in graph theory~\cite{Menger1927,Gallai1961,DBLP:journals/jct/RobertsonS95}.
The best known example is the case of cycles: we want to relate the size of a largest family, or \emph{packing}, of vertex-disjoint cycles in a graph, with the size of a smallest set of vertices intersecting all cycles (a \emph{feedback vertex set}).
Clearly, the existence of a large cycle packing is evidence that we cannot hit all cycles with few vertices.
However, it is far from obvious whether this is the only structure that forces feedback vertex set to be large.
By a celebrated result of Erd\H{o}s and P\'osa~\cite{DBLP:journals/cjm/ErdosP65}, this is indeed true:
	Every graph $G=(V,E)$ contains either $k$ pairwise disjoint cycles or a set $X \subseteq V$ of size $\Oh(k \log k)$ such that $G-X$ has no cycles.
Furthermore, the function $\Oh(k \log k)$ governing the gap between packing and hitting outcomes turns out to be asymptotically optimal.

This
sparked a flurry of results of similar flavour, now dubbed \emph{\EP properties}, for variations of the initial problem.
Two main variations emerging from the literature consist in altering the \emph{objects} being packed (cycles), and altering the relation imposed on the packing (disjointness).
The interplay between possible weakenings and strengthenings of the two uncovers an intricate boundary of where the \EP property holds. 

If we still ask for a disjoint packing, and wish to generalize the graphs being packed, we can not only consider graphs other than cycles, but also alter their containment relation. 
One of the most general results in this vein, due to Robertson and Seymour~\cite{DBLP:journals/jct/RobertsonS86}, is that planar graphs have the \EP property, with respect to the minor relation.
That is, for every planar graph $H$, there exists a function $f_H : \N \to \N$ such that every graph $G$ either has $k$ vertex-disjoint subgraphs, each containing $H$ as a minor, or a set $X$ of $f_H(k)$ vertices such that $G-X$ is $H$-minor-free.
Furthermore, such a statement is false if $H$ is not planar.

Following the original motivation of Erd\H{o}s and P\'{o}sa, a long-studied line of research studies \EP properties for variants of cycles, in particular, long cycles and holes.
Regarding the former, it is easily derived from~\cite{DBLP:journals/jct/RobertsonS86} that for any $t$, cycles of length at least $t$ satisfy the \EP property.
Then, obtaining the best governing function (asymptotically) was an active effort~\cite{DBLP:journals/jgt/Thomassen88,DBLP:journals/combinatorica/BirmeleBR07}, culminating in an optimal $\Oh(kt + k \log k)$ due to Mousset, Noever, Škorić, and Weissenberger~\cite{MOUSSET201721}.
The case of holes, i.e., induced cycles of length at least~$4$, was investigated by Kim and Kwon~\cite{DBLP:journals/jct/KimK20}, who showed that they do indeed satisfy the \EP property.
On the negative side, they ruled out this possibility for induced cycles of length at least $5$. 

If we now wish to alter the packing relation, there are two directions in which one can go.
On the one hand, there are countless examples of objects that do not satisfy the \EP property, begging for weakened packing notions.
This is notably the case for non-planar minors or parity-constrained cycles, where the \EP property has been recovered by asking for \say{half-integral} packings, see Reed~\cite{reed1999mangoes} and Liu~\cite{liuMinorHalfInt}.
Of particular relevance to our problem is the fact that no variation of the packing relation is known to recover the \EP property for long holes, and it is known that the fractional packing relaxation fails~\cite{DBLP:journals/jct/KimK20}. 
On the other hand, for objects that do satisfy the \EP property, one should look to strengthen the packing relation.
It is natural then to ask not only for disjointness, but also non-adjacency. Formally, an \emph{induced packing} of some object (for any containment relation), is a packing of pairwise disjoint and non-adjacent copies of it.
Then, an \emph{induced \EP property} gives a counterpart to the non-existence of an induced packing.

With such a strong packing notion, what should the corresponding notion of a hitting set be? 
It is easy to see that hitting with individual vertices is usually not enough,%
\footnote{In most cases, blowing up every vertex of a graph into a large independent set or a clique 
does not change the packing number but inflates the hitting set size.}
and the first natural weakening is to hit with closed neighborhoods, i.e., radius-1 balls (which is canonical in the induced realm~\cite{hickingbotham2025induced}).
That is, we should ask for a few vertices $X$, such that their closed neighborhood $N[X]$ hits all copies of cycles.
This question sits among many problems aiming at \textsl{induced} analogues of results known for subgraphs or minors, see e.g.~\cite{KORHONEN2023206,DBLP:phd/us/Gartland23}.
This is in particular exemplified by the conjectured induced analogue of Menger's theorem~\cite{hickingbotham2025induced}, asking for many pairwise non-adjacent paths, or a small set of vertices whose neighborhoods hit all paths.

Regarding the induced \EP property, the following was shown for long cycles.
\begin{theorem}[Ahn, Gollin, Huynh, Kwon~\cite{DBLP:conf/soda/AhnGHK25}]\label{thm:coarseEP}
	Let $k \geq 1,t \geq 3$ be integers.
	Every graph $G=(V,E)$ contains either:
	\begin{itemize}
	\item $k$ pairwise disjoint and anti-adjacent cycles, each of length at least $t$, or
	\item a set $X \subseteq V$ of size $\Oh(t \cdot k \log k)$ such that $G-N[X]$ has no induced cycle of length at least $t$.
	\end{itemize}
\end{theorem}

We emphasize that the cycles found in \cref{thm:coarseEP} are non-necessarily induced. 
This leaves the following question open:
\emph{
Do long induced cycles have the induced \EP property?
}

This question can be conveniently stated in terms of \emph{induced minors}.
A graph $H$ is an induced minor of a graph $G$, if it can be obtained from $G$ by deleting vertices and contracting edges.
We say that $G$ is \emph{$H$-induced-minor-free} if it does not contain $H$ as an induced minor.

With this terminology, we can formulate the following conjecture, of an induced minor analogue to the aforementioned result of Robertson and Seymour for minors~\cite{DBLP:journals/jct/RobertsonS86}.
(For a connected graph $H$, we denote by $kH$ the graph consisting of $k$ disjoint copies of $H$.)

\begin{restatable}{conjecture}{conjEP}
	\label{conj:EP-planar}
	For every planar graph $H$, there exists a function $f: \N \to \N$ such that
	every graph $G=(V,E)$ contains either:
	\begin{itemize}
	\item $k H$ as an induced minor, or
	\item a set $X \subseteq V$ of size at most $f(k)$ such that $G-N[X]$ is $H$-induced-minor-free.
	\end{itemize}
\end{restatable}
\noindent
\cref{conj:EP-planar} is a strengthening of a conjecture of Ahn, Gollin, Huynh, Kwon~\cite[Conjecture 8.5]{DBLP:conf/soda/AhnGHK25} whose statement only requires that $G$ becomes $H$-induced-minor-free after deleting $f(k)$ balls of some constant radius.

Dujmović, Joret, Micek, and Morin~\cite{dujmovic2024erd}
recently answered a question of Georgakopoulos and Papasoglu~\cite{DBLP:journals/combinatorica/GeorgakopoulosP25}
and settled a coarse analog of \cref{conj:EP-planar} for cycles, 
establishing the \EP property for packing cycles at pairwise distance greater than $d$
versus hitting cycles with radius-$19d$ balls. 
Note that the $d=1$ case is exactly the induced packing case.


\medskip

With this background in mind, we proceed to present our results.

\paragraph{Induced \EP for long holes.} 
As the first result, we confirm \cref{conj:EP-planar} for the case where $H$ is any cycle, answering the aforementioned question about the induced \EP property for long induced cycles.
Furthermore, our proof yields a polynomial-time algorithm that returns one of the two outcomes in the statement of \cref{conj:EP-planar}.

\begin{restatable}[Algorithmic induced Erd\H{o}s--P\'osa for long holes]{theorem}{EPCyclesAlgo}\label{thm:EP-cycles-algo}
	There exists an algorithm that, given a graph $G$ and an integer $t$, in time $|V(G)|^{\Oh(t)}$, outputs 
	an integer $k$ and the following objects:
	\begin{itemize}[itemsep=0px]
		\item a family of $k$ vertex-disjoint anti-adjacent induced cycles, each of length at least $t$ (i.e., an induced minor model of $kC_t$ in $G$); and
		\item a set $X \subseteq V(G)$ of size $\Oh(tk\log k)$ such that $G - N[X]$ does not contain any induced cycle of length at least $t$ (i.e., an induced minor model of $C_t$).
	\end{itemize}
\end{restatable}
\noindent We remark that a slight modification
of the arguments of~\cite{DBLP:conf/soda/AhnGHK25}
(take a $2$-subdivision of the lower bound construction for non-induced $C_{\geq t}$ of~\cite{MOUSSET201721})
yields a $\Omega(tk + k \log k)$ lower bound for~\Cref{thm:EP-cycles-algo}.
This leaves a small room for possible improvement.

\paragraph{Induced \EP for Long thetas and three-path-configurations.}

Actually, \cref{thm:EP-cycles-algo} can be derived from an analogous statement for another type of objects, namely, \emph{long three-path-configurations}.
Let us introduce some notation first; see \cref{fig:lpcs} for an illustration.

By $C_{\geq t}$ we denote a cycle of length at least $t$.
A \emph{theta}, denoted by $\Theta$, is any (multi)graph obtained from the two-vertex multigraph $\Theta_0$ with three parallel edges by a series of edge subdivisions. 
A $\Theta_t$ is a $\Theta$ where each path
is of length exactly $t+1$ (so that they have exactly $t$ internal vertices each).
A $\Theta_{\geq t}$ is defined analogously with ``exactly'' replaced by ``at least.''
Note that for every $t \geq 1$, every $\Theta_{\geq t}$ contains $\Theta_t$ as an induced minor.

A \emph{$t$-long three path configuration ($t$-long 3PC)}, denoted $\Xi_{\geq t}$, is a set of three induced paths, each of length 
more than $t$ (i.e., having at least $t$ internal vertices), that are pairwise disjoint and anti-adjacent, except for the following:
one can specify which endpoint is a start and which is an end in such a manner, that all three starts
are either equal or pairwise distinct and adjacent, and the same holds for all three ends.

\begin{figure}[t]
\centering
\includegraphics[scale=0.7]{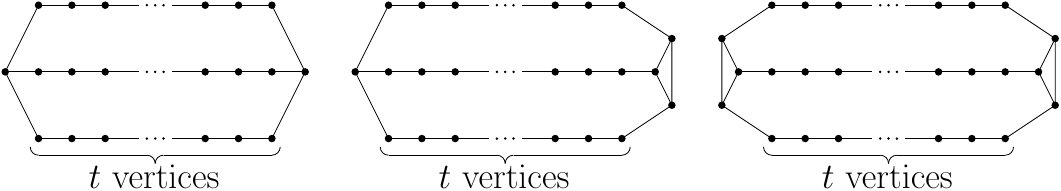}
\caption{Three minimal (with respect to the number of degree-2 vertices) $\Xi_{\geq t}$s.
The first one is $\Theta_t$.}
\label{fig:lpcs}
\end{figure}

We have the following lemma saying that $\Xi_{\geq t}$ is \emph{almost} an induced minor model of a $\Theta_t$. A detailed proof for it can be found in \cref{sec:theta-xi-section}.

\begin{restatable}{lemma}{obsThetaXi}
	\label{lem:theta-xi}
	Let $t$ be a positive integer.
	\begin{enumerate}[itemsep=0pt]
		\item Every induced minor model of a $\Theta_{\geq t+2}$ contains a $\Xi_{\geq t}$.
		\item Every $\Xi_{\geq t}$ contains $\Theta_{\geq t}$ as an induced minor.
	\end{enumerate}	
	Furthermore, given the assumed structure, the claimed structure can be found in polynomial time.
\end{restatable}

A $t$-long 3PC is almost like an induced minor model of a $\Theta_{\geq t}$, but there are some subtle differences
in corner cases (see \cref{fig:turtle}).
In particular, it is relatively easy to check if a given graph contains a 
$\Xi_{\geq t}$ in time $n^{\Oh(t)}$ (this is formalized later as \cref{lem:detect-L3PC}), while the complexity status of checking whether a graph contains
$\Theta_{t}$ as an induced minor remains open, with the $t=1$ case solved recently~\cite{DBLP:conf/iwoca/DallardDHMPT24}. This is the main reason we will be able to obtain an algorithmic
Erd\H{o}s--P\'osa result for $\Xi_{\geq t}$, but only an existential one for $\Theta_{t}$ induced minors.

\begin{restatable}[Induced Erd\H{o}s--P\'osa for long thetas]{theorem}{EPTheta}\label{thm:EP-thetas}	
	For every graph $G$ and an integer $t \geq 1$, 
	if $k$ is the maximum integer such that
	$G$ contains $k \Theta_t$ as an induced minor, 
	then $G$ contains a set $X \subseteq V(G)$ of size $\Oh(tk \log k)$ such that $G-N[X]$ is $\Theta_t$-induced-minor-free.
\end{restatable}

\begin{restatable}[Algorithmic induced Erd\H{o}s--P\'osa for long 3PCs]{theorem}{algEPforXis}\label{thm:EP-3LPC-algo}
	There exists an algorithm that, given a graph $G$ and an integer $t$, in time $|V(G)|^{\Oh(t)}$, outputs an integer $k$ and the following objects:
	\begin{itemize}[itemsep=0px]
		\item a family of $k$ vertex-disjoint anti-adjacent $\Xi_{\geq t}$s in $G$; and
		\item a set $X \subseteq V(G)$ of size $\Oh(tk\log k)$ such that $G - N[X]$ does not contain any $\Xi_{\geq t}$.
	\end{itemize}
\end{restatable}

In \cref{lem:ep-cycles-no-xi}, we show that in the class of graph without 
a $\Xi_{\geq t}$, the cycles $C_{\geq t}$ have the induced \EP{} property
with the size of $X$ bounded by $\Oh(tk)$.
Consequently, 
the lower bound of $\Omega(tk + k \log k)$ extends also to~\cref{thm:EP-thetas,thm:EP-3LPC-algo}.

Note that even though $\Theta_t$ and $\Xi_{\geq t}$ are closely related by \cref{lem:theta-xi} and both contain $C_{\geq t}$, the results of \cref{thm:EP-cycles-algo,thm:EP-thetas,thm:EP-3LPC-algo} are not immediate consequences of each other.
However, our approach actually allows us to derive \cref{thm:EP-cycles-algo,thm:EP-thetas} from \cref{thm:EP-3LPC-algo}.

\paragraph{Dominated balanced separators.}

It turns out that our approach can be used to derive some results outside the realm of the (induced) Erd\H{o}s--P\'osa property.
Let $G$ be a graph with vertex weights $\wei: V(G) \to \mathbb{R}_{\geq 0}$. For a set $X \subseteq V(G)$, we write $\wei(X) = \sum_{x \in X} \wei(x)$.
A \emph{balanced separator} in $(G, \wei)$ is a set $S \subseteq V(G)$ such that every component of $G-S$ has weight at most $\wei(V(G))/2$. We say that a class $\mathcal{G}$ of graphs \emph{admits $d$-dominated balanced separators}, for some integer $d$, if for every $G \in \mathcal{G}$ and every weight function $\wei$, there exists $X \subseteq V(G)$ of size at most $d$ such that $N[X]$ is a balanced separator in $(G,\wei)$.

We show the following result.

\begin{restatable}{theorem}{thmdbsfinal}\label{thm:dbs-final}
	For every $k,t \geq 1$, the following classes of graphs admit $O(t k \log k)$-dominated balanced separators:
	\begin{itemize}
		\item $kC_{t}$-induced-minor-free graphs,
		\item $k\Theta_{t}$-induced-minor-free graphs,
		\item graphs with no $k$ vertex-disjoint and anti-adjacent $\Xi_{\geq t}$s.
	\end{itemize}
	Furthermore, given a vertex-weighted graph $G$ on $n$ vertices,
	in time $n^{\Oh(tk \log k)}$ one can either find a $O(t k \log k)$-dominated balanced separator,
	or a family of $k$ vertex-disjoint and anti-adjacent $\Xi_{\geq t}$s.
	From the latter outcome one can, in polynomial time, obtain an induced minor model of $k\Theta_t$ and $kC_t$.
\end{restatable}

Recall that the complexity of finding $\Theta_t$ as an induced minor is unclear.
Still, \cref{thm:dbs-final} circumvents this issue by either returning an induced minor model,
or balanced separator dominated by few vertices.

Again, \cref{thm:dbs-final} confirms a special case of a well-known conjecture, this time by Gartland and Lokshtanov~\cite{DBLP:phd/us/Gartland23}.

\begin{restatable}[Gartland and Lokshtanov~\cite{DBLP:phd/us/Gartland23}]{conjecture}{conjdbs}
	\label{conj:dbs-planar}
	For every planar graph $H$, there exists a constant $d$ 
	such that $H$-induced-minor-free graphs admit $d$-dominated balanced separators.
\end{restatable}

\cref{conj:dbs-planar} is wide open and so far was confirmed only for very restricted classes of planar graphs $H$, 
including paths~\cite{DBLP:journals/algorithmica/BacsoLMPTL19}, cycles~\cite{DBLP:journals/siamcomp/ChudnovskyPT24}, wheels~\cite{DBLP:journals/corr/abs-2512-12329}, $K_{2,q}$ for any fixed $q$, or $K_5$ minus an edge~\cite{DBLP:journals/jct/DallardMS24}.

\cref{thm:dbs-final} has some interesting consequences.
To discuss the first one, we need to introduce the concept of \emph{tree-independence number}. The \emph{independence number} of a tree decomposition is the maximum size of an independent set contained in a single bag. The \emph{tree-independence number} of a graph $G$ is the minimum independence number of a tree decomposition of $G$.
Graphs with small tree-independence number attract a significant attention, both from the structural~\cite{DBLP:journals/jct/DallardMS24,DBLP:conf/soda/ChudnovskyGHLS25,DBLP:journals/jctb/ChudnovskyHLS24} and the algorithmic~\cite{DBLP:journals/jcss/LimaMMORS26,DBLP:conf/esa/LokshtanovPR26} point of view. One of well-known open problems in the area is the following conjecture. (A graph is $H$-free if it does not contain $H$ as an induced subgraph.)

\begin{conjecture}[Dallard, Krnc, Kwon, Milani\v{c}, Munaro, \v{S}torgel, and Wiederrecht~\cite{DBLP:journals/corr/abs-2402-11222}]\label{conj:treealpha}
For every integer $\ell$ and every planar graph $H$, there exists an integer $c_{\ell,H}$
such that every graph  which is $H$-induced minor-free and $K_{1,\ell}$-free
has tree independence number at most $c_{\ell,H}$.
\end{conjecture}

Note that if a graph $G$ is $k \Theta_t$-induced-minor-free and $K_{1,\ell}$-free, then the largest independent set in the  subgraph induced by the balanced separator given by \cref{thm:dbs-final} is of size at most $\Oh(t\ell \cdot k \log k)$.
Combining this observation with known results concerning building tree decompositions of small tree-independence number from separators with no large independent sets~\cite[Lemma 7.1]{DBLP:journals/jctb/ChudnovskyHLS24}, we immediately obtain the following corollary, confirming \cref{conj:treealpha} in the case that every component of $H$ is a theta.

\begin{corollary}\label{cor:thetas-treealpha}
	For every $k,t$, and $\ell$, every graph that is $k \Theta_t$-induced-minor-free and $K_{1,\ell}$-free has tree-independence number $\Oh(t\ell \cdot k \log k)$.
\end{corollary}

\paragraph{Algorithmic consequences.}

\cref{thm:dbs-final} yields also some algorithmic corollaries.
First, combining it with known techniques~\cite{DBLP:journals/dam/GroenlandORSSS19,DBLP:journals/algorithmica/BacsoLMPTL19,DBLP:journals/siamcomp/ChudnovskyPT24,DBLP:journals/corr/abs-2512-12329}, we immediately obtain the following result.

\begin{theorem}\label{thm:thetas-algos}
	Let $k,t$ be fixed constants.
	The \textsc{Max Weight Independent Set} (\MWIS) problem admits a QPTAS in $k \Theta_t$-induced-minor-free graphs.	
	Furthermore, \MWIS and \textsc{List-3-Coloring} can be solved in time $2^{\Oh(\sqrt{n \log n})}$ in $n$-vertex $k \Theta_t$-induced-minor-free graphs.

	In all cases, the algorithm can accept any graph on input and, within the claimed running time bound, produce the desired output or an induced minor model of $k \Theta_t$.
\end{theorem}

Actually, we can also tackle a more general problem.
For fixed integer $r$ and a \textsf{CMSO}$_2$\footnote{\textsf{CMSO}$_2$ is a logic where one can use vertex, edge, and (vertex or edge) set variables, check vertex-edge incidence, quantify over variables, and apply counting predicates modulo fixed integers. See~\cite{DBLP:books/sp/CyganFKLMPPS15} for a formal introduction.} formula $\psi$,
in the $(\tw \leq r,\psi)$-\MWIS problem we are given a vertex-weighted graph $G$ and we aim to find a maximum-weight induced subgraph of treewidth at most $r$ that satisfies $\psi$, or conclude that no such subgraph exists.
This formalism captures a wide range of problems, including \MWIS, \textsc{Induced Matching}, \textsc{Feedback Vertex Set}, or \textsc{Even Cycle Transversal}.

Combining \cref{thm:thetas-algos} with the approach based on the so-called \emph{blob graphs}~\cite{DBLP:conf/stoc/GartlandLPPR21},
we obtain a QPTAS for \emph{unweighted} $(\tw \leq r,\psi)$-\MWIS under a minor restriction that $\psi$ is \emph{hereditary}, i.e., closed under vertex deletion and disjoint unions.

\begin{restatable}{theorem}{thmcmso}\label{thm:thetas-algo-cmso}
Let $r \geq 0$, let $k,t$ be positive integers, $\varepsilon \in (0,1)$ be a real, and $\psi$ be a hereditary \textsf{CMSO}$_2$ formula.
There is an algorithm that, given a graph $G$, in quasipolynomial time, returns one of the following outputs:
\begin{itemize}
\item an induced minor model of $k\Theta_t$ in $G$, or
\item a solution to $(\tw \leq r,\psi)$-\MWIS of size at least $(1-\varepsilon)$ times the optimum, or
\item a correct conclusion that no solution to $(\tw \leq r,\psi)$-\MWIS exists.
\end{itemize}
\end{restatable}

This is a step towards confirming the following algorithmic counterpart of \cref{conj:dbs-planar}.

\begin{restatable}[Gartland, Lokshtanov~\cite{DBLP:phd/us/Gartland23}]{conjecture}{conjalg}\label{conj:alg}
For every planar graph $H$, every problem expressible as $(\tw \leq r,\psi)$-\MWIS is polynomial-time-solvable in $H$-induced-minor-free graphs.
\end{restatable}

We remark that even weakenings of this conjecture, for example, allowing a quasipolynomial-time algorithm, or even a QPTAS, are wide open.
Approximation schemes from \cref{thm:thetas-algos} and \cref{thm:thetas-algo-cmso} extend recent analogous results for $k C_{t}$-induced-minor-free graphs~\cite{swatpaper}, obtained via a different approach.

\paragraph{Outline.}
Throughout the paper we use standard graph-theoretic notation\onlysoda{ (see \cref{sec:prelims})}.

As an overview of our main techniques, in Sections~\ref{sec:models} and~\ref{sec:cycles} 
we present the proof of \cref{thm:EP-cycles-algo}.
Section~\ref{sec:models} introduces the notion of a short $t$-model that is the key ingredient
also in the proof of \cref{thm:EP-3LPC-algo}. \onlysoda{(Some tedious technical proofs are postponed to Appendix~\ref{sec:models-later}.)}
Section~\ref{sec:cycles} proves \cref{thm:EP-cycles-algo}.
Section~\ref{sec:cycles} contains also a sketch how to deduce \cref{thm:EP-cycles-algo} from \cref{thm:EP-3LPC-algo}.

We rely on the following (corollary of the) classical result of Simonovits~\cite{Simonovits67}.

\begin{restatable}{theorem}{packcyclesinH}\label{thm:simonovits}
    There exists a function $\fsim(k) = \Oh(k \log k)$ such that every subcubic multigraph $G$ with minimum degree at least $2$ and at least $\fsim(k)$ vertices of degree $3$ contains $k$ vertex-disjoint cycles (as a subgraph).
	Furthermore, such a packing can be found in polynomial time.
\end{restatable}

We prove the main result, \cref{thm:EP-3LPC-algo}, in Sections~\ref{sec:prelims}--\ref{sec:L3PC}.
Section~\ref{sec:prelims} introduced further preliminaries and notation,
including the proof of \cref{lem:theta-xi}.
Section~\ref{sec:pack-in-multigraph} presents a statement analogous to \cref{thm:simonovits}
for packing thetas. Finally, in Section~\ref{sec:L3PC} we prove \cref{thm:EP-3LPC-algo}.

Next, we deduce corollaries of \cref{thm:EP-3LPC-algo}: in Section~\ref{sec:L3PC-cor}, we deduce \cref{thm:EP-thetas} from \cref{thm:EP-3LPC-algo},
Section~\ref{sec:dbs} contains the proof of \cref{thm:dbs-final},
and in Section~\ref{sec:blobs} we present the proof of \cref{thm:thetas-algo-cmso}.
Finally, in Section~\ref{sec:outro} we discuss open problems and future research directions.

	\section{Models}\label{sec:models}

\subsection{Definitions}

We define $t$-models, which capture induced minor models of a $t$-subdivided graph, where subdivision paths are actual induced paths of $G$ (see also~\Cref{fig:ear-addition}).
Two vertex sets \emph{do not touch} if they are disjoint and anti-adjacent.
\begin{definition}[$t$-model]\label{def:model}
  Given a graph $G$ and a multigraph $H$, a \emph{$t$-model} of $H$ in $G$ is a tuple $\mH = (H,\eta)$, where the \emph{bag function} $\eta : V(H) \cup E(H) \mapsto 2^{V(G)}$ satisfies the following.

  \begin{enumerate}[label=(\roman*),itemsep=0pt]
    \item For any $v \in V(H)$, $\eta(v)$ induces a connected subgraph of $G$, and for any $u \in V(H)$, $u \neq v$, the sets $\eta(u)$ and $\eta(v)$ do not touch.
    \item For any $e = uv \in E(H)$, $\eta(e)$ induces a path of length more than $t$ in $G$ (i.e., with at least $t$ internal vertices)
    with one endpoint $\eta(e, u) \in \eta(e) \cap \eta(u)$ and the other endpoint
    $\eta(e,v) \in \eta(e) \cap \eta(v)$.
    \item For any $v \in e \in E(H)$, $\eta(e) \setminus \{\eta(e,v)\}$ 
    and $\eta(v) \setminus \{\eta(e,v)\}$ do not touch.
    \item For any $e \in E(H)$ and $w \in V(H)$, if $w \notin e$, then
    the sets $\eta(e)$ and $\eta(w)$ do not touch.
    \item For any $e, f \in E(H)$, $e \neq f$,
    the sets $\eta(e)$ and $\eta(f)$ do not touch, except for the following possibility: if $e$ and $f$ have a common endpoint $v$
    then we allow $\eta(e,v) = \eta(f,v)$ or $\eta(e,v)\eta(f,v) \in E(G)$. 
  \end{enumerate}
\end{definition}
Note that $G$ can be a multigraph, and hence $e$ and $f$ in the last point can be two parallel edges.

Let $\mH = (H,\eta)$ be a $t$-model of $H$ in $G$.
Let $e = uv \in E(H)$.
The length-$\ell$ prefixes of $\eta(e)$ starting at $\eta(e,u)$ and $\eta(e,v)$ are denoted $\eta_{\ell}(e,u)$ and $\eta_{\ell}(e,v)$ respectively;
$\eta_\ell(e,v) = \eta_\ell(e,u) = \eta(e)$ if the length of $\eta(e)$ is at most $\ell$.
We denote by $\eta(e)[a,b]$ the subpath of $\eta(e)$ between $a,b \in \eta(e)$.

We let $V(\mH)$ be the union of all vertex and edge bags, and,
when we consider $\mH$ as a subgraph of $G$,
we mean the subgraph induced by $V(\mH)$.
We shorten $N(V(\mH))$ to $N(\mH)$.
We have the following immediate observation from the definition.
\begin{observation}\label{obs:model-subdivision}
If $\mH$ is a $t$-model of $H$ in $G$, then an induced minor model of
a $t$-subdivision of $H$ can be constructed from $G[V(\mH)]$
by contracting every vertex bag into a single vertex and contracting every edge
bag to a path of length exactly $t+1$.
\end{observation}
An immediate corollary of Observation~\ref{obs:model-subdivision} is that a packing
of disjoint cycles in $H$ projects to a packing of anti-adjacent long holes in $G$.
(\cref{lem:xis-from-model} is an analog of this statement that extracts anti-adjacent $\Xi_{\geq t}$s from a packing of disjoint
$\Theta$ models in $H$.)
\begin{corollary}\label{cor:cycles-from-model}
  If $\mH$ is a $t$-model of $H$ in $G$ and $H$ admits a packing of $k$
  vertex-disjoint cycles, then $G[V(\mH)]$ admits a packing of $k$
  anti-adjacent $C_{\geq t}$s.
\end{corollary}

A \emph{guarded $t$-model} consists of a $t$-model $\mH = (H,\eta)$ and 
a set $R \subseteq V(\mH)$ of \emph{guards} or \emph{reserved vertices}.
Somewhat abusing the notation, we will denote a guarded $t$-model
as a triple $\mH = (H,\eta,R)$ and still use the notation $V(\mH)$ for the union of the values of $\eta$.

We define ears of the model as the paths only touching $\mH$ at their endpoints, while avoiding $N[R]$.
\begin{definition}[ear]\label{def:ear}
Given a guarded $t$-model $\mH = (H,\eta,R)$ in $G$, an \emph{ear} of $\mH$ is an induced path $P = p_1,...,p_k$ of $G$ such that $p_1,p_k \in N(\mH)-N[R]$, and for any $i \in [2,k-1]$,  $p_i \notin N[\mH]$. 
\end{definition}
\noindent
When $N(p_1) \cap \mH \subseteq \eta(e)$ and $N(p_k) \cap \mH \subseteq \eta(f)$ for some $e,f \in E(H)$, we say that $P$ is an ear \textsl{between} $e$ and $f$.
The \emph{distance} of an ear $P = p_1,p_2,\ldots,p_k$ in $\mH$ is the minimum distance in $G[V(\mH)]$ between $N(p_1) \cap \mH$ and $N(p_k) \cap \mH$.
The \emph{length} of an ear $P$ is defined naturally as the length of the path $P$. 

We now define short $t$-models.
\begin{definition}[short $t$-model]\label{def:short-model}
  For any $t \geq 1$, a \emph{short $t$-model} of $H$ in $G$ is a guarded $t$-model $\mH = (H,\eta,R)$ where:
    \begin{itemize}[itemsep=0pt]
        \item For any $v \in V(H)$, $|\eta(v)| \leq 4$ and $\eta(v) \subseteq R$. For any $e = uv \in E(H)$, $\eta_{t+3}(e,u) \cup \eta_{t+3}(e,v) \subseteq R$.
        \item For any $v \in N(\mH) \setminus N[R]$, set $N(v) \cap \mH$ is contained in a $3$-vertex subpath of some $\eta(e)$. In particular, any ear is between (possibly equal) edges of $H$.
        \item Any ear of distance at least $t+3$ in $\mH$ is of length more than $t$.
    \end{itemize}
\end{definition}
\noindent
Observe that, due to the first point, any ear connecting two different edges of $H$
is of distance at least $t+3$, and hence, due to the last point, of length more than $t$.

In the context of \cref{thm:EP-cycles-algo}, we observe that finding a shortest $C_{\geq t}$ in a graph
gives us a short $t$-model of a single vertex with a single edge (loop),
with an annoying corner case of the shortest $C_{\geq t}$ being actually a $C_t$.
\onlysoda{A formal proof can be found in Section~\ref{sec:models-later}, while}
\cref{lem:detect-L3PC} in Section~\ref{sec:L3PC} is an analog of this lemma in the context of $\Xi_{\geq t}$s,
where we produce a short $t$-model of $\ThetaZero$.
\begin{restatable}{lemma}{leminitcycle}\label{lem:init-cycle}
  Given an $n$-vertex graph $G$ and an integer $t$, one can in time $n^{\Oh(t)}$
  either correctly conclude that $G$ does not contain a $C_{\geq t}$,
  find an induced $C_t$ in $G$, or find a short $t$-model $\mH=(H,\eta,R)$
  in $G$ where $H$ is a one-vertex graph with one edge being a loop and $|R| \leq 2t+7$.
\end{restatable}
\onlyarxiv{\begin{proof}
  First, in $n^{\Oh(t)}$ time find a shortest $C_{\geq t}$ in $G$:
  iterate over all choices $x_1,\ldots,x_t$ of $t$ consecutive vertices on the sought hole
  and close the candidate hole with a shortest path from $x_t$ to $x_1$ 
  in $G-(N[\{x_2,x_3,\ldots,x_{t-1}\}] \setminus \{x_1,x_t\})$. 
  If no hole is found, report the first outcome. If a hole of length $t$ is found, 
  report it as the second outcome.

  Otherwise, let $C$ be the hole found. Let $x,y,z$ be arbitrary consecutive vertices of $C$.
  Let $H$ be a graph with a single vertex $v$ and single edge $e = vv$. 
  Define $\eta(v) = \{x,y,z\}$, $\eta(e) = C \setminus \{y\}$, and $R$
  consisting of all vertices of $C$ that are within distance at most $t+3$ from $y$ on $C$.
  Clearly, $\mH = (H,\eta,R)$ is a guarded $t$-model and $|R| \leq 2t+7$. 

  It remains to show that $\mH$ is short. 
  To this end, consider first $v \in N(\mH) \setminus N[R]$.
  Assume that there are two neighbors $u_1,u_2 \in N(v) \cap \eta(e)$ within distance
  more than $2$ on $C$; without loss of generality, pick $u_1$ and $u_2$ such that $\eta(e)[u_1,u_2]$
  is maximal. 
  Note that we can shortcut $C$ via $u_1-y-u_2$, obtaining a shorter induced cycle, 
  but still longer than $t$ as $R$ contains $2t+7$ vertices. This is a contradiction with the choice of $C$.
  Hence, for every $v \in N(\mH) \setminus N[R]$, $N(v) \cap \eta(e)$ is contained in a subpath
  of $\eta(e)$ of length at most $2$.

  Assume now that $P$ is an ear of $\mH$ of distance more than $t+3$.
  Clearly, $P$ connects $e$ with itself. Let $x_1$ and $x_2$ be the endpoints of $P$, and for $i=1,2$ let $y_i \in \eta(e) \cap N(x_i)$ be the vertex such that one of the subpaths of
  $C$ between $y_i$ and $y$, denoted henceforth $P_i$, does not contain any other vertex of $N[P]$.
  Then, the concatenation of $P_1$, $P$, and $P_2$ is a hole of length at least $|R| \geq t$.
  As $C$ is a shortest $C_{\geq t}$ and $P$ is of distance at least $t+3$, we infer
  that $P$ is of length at least $t$. This proves that $\mH$ is short.
\end{proof}}

\subsection{Ear addition and ear-decompositions}

We now show how to augment a short $t$-model with an ear such that the resulting model is also short. Fix an integer $t \geq 1$.
\begin{definition}[ear addition]\label{def:ear-addition}
    Consider a short $t$-model $\mH = (H,\eta,R)$, and any ear $P = p_1,\ldots,p_k$ of distance at least $t+3$ on $\mH$, say between edges $e,f \in H$.
    We define the operation of adding $P$ to $\mH$, yielding a short $t$-model $\mH' = (H',\eta',R')$.
    
    If $e=f$, then $H'$ is obtained from $H$ by subdividing $e=uv$ into the path $u x y v$ (or a cycle $u x y u$ if $u=v$ and $e$ is a self-loop), letting $g$ be the resulting edge between $x$ and $y$, and adding a parallel edge $g'$ between $x$ and $y$.
    Otherwise (in particular, if $e$ and $f$ are parallel), $H'$ is obtained from $H$ by subdividing $e$ into $u x v$, and $f$ into $w y z$, and adding the edge $g' = x y$.
    
    Let $e = uv$, and $p_1^u,p_1^v$ be the neighbors of $p_1$ closer to $u,v$ on $\eta(e)$ respectively.
    Let $f = wz$, and $p_k^w,p_k^z$ be the neighbors of $p_k$ closer to $w,z$ on $\eta(f)$ respectively.

    In both cases, the new vertices are given bags $\eta'(x) = p_1 \cup \eta(e)[p_1^u,p_1^v]$ and $\eta'(y) = p_k \cup \eta(f)[p_k^w,p_k^z]$, and we let $\eta'(g') = P$.
    The remaining new edges are defined as follows.
    \begin{itemize}[nosep]
        \item If $e=f$, we split $\eta(e)$ into $\eta'(ux) = \eta(e)[\eta(e,u),p_1^u]$; $\eta'(g) = \eta(e)[p_1^v,p_k^w]$ and $\eta'(y v)=\eta(e)[p_k^z,\eta(e,v)]$.
        \item If $e \neq f$, we split $\eta(e)$ into $\eta'(ux) = \eta(e)[\eta(e,u),p_1^u]$ and $\eta'(x v) = \eta(e)[p_1^v,\eta(e,v)]$; and symmetrically split $\eta(f)$ into $\eta'(w y) = \eta(f)[\eta(f,w),p_k^w]$ and $\eta'(y z) = \eta(f)[p_k^z,\eta(f,z)]$.
    \end{itemize}
    All remaining vertices and edges of $H'$ are given the same bags as in $\mH$. Then, we let $R'$ be the union of $R$, $\eta'(x) \cup \eta'(y)$, and the length $t+3$ prefixes of $\eta'(e')$ from $x$ and $y$ for every $e' \in E(H') \setminus E(H)$.
\end{definition}

\begin{figure}
  \centering
  \includegraphics[scale=1.4]{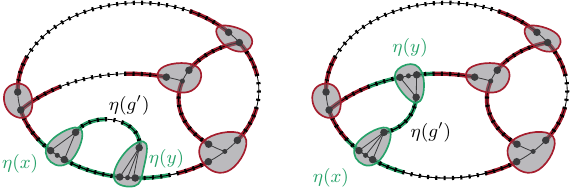}
  \caption{Adding an ear to a short $t$-model, creating vertices $x$ and $y$ and an edge $g'$ between them. Shown here are the two possible cases of ear addition: with endpoints on the same edge of $H$
  (left) and with endpoints on two distinct edges (right). The grey blobs depict vertex bags and the highlighted
  parts of paths depict vertices that are mandatorily in the guard set $R$.}\label{fig:ear-addition}
\end{figure}

It is easy to verify that adding an ear to a short $t$-model according to~\Cref{def:ear-addition} indeed yields a (non-necessarily short) $t$-model.
Furthermore, with a bit more tedious work, we prove
that adding a \emph{shortest} ear (in a particular sense, stated below)
preserves shortness of the model.
\onlysoda{The proofs of the next two lemmata are postponed to Section~\ref{sec:models-later}.}

\begin{restatable}{lemma}{lemmodeladdear}\label{lem:model-add-ear}
  Adding an ear to a short $t$-model produces a $t$-model.
\end{restatable}
\onlyarxiv{\begin{proof}
\begin{enumerate}[nosep, label=(\roman*)]
    \item In both cases, $\eta'(x)$ and $\eta'(y)$ are connected, they do not touch, and do not touch any previous vertex bag, since those are contained in $R$.
    \item Since $P$ is of distance at least $t+3$ on $\mH$, 
    \begin{enumerate}[nosep]
      \item in the case $e=f$, $\eta'(g)$ is of length at least $t+3$, and 
      \item in both cases, because $\mH$ is short, $\eta'(g')$ is of length more than $t$. 
    \end{enumerate}
    \item All other new edge bags contain paths of length at least $t+3$, 
    because $R$ contains the length $t+3$ prefixes of $\eta(e)$ and $\eta(f)$.
    \item The fact that new edge bags touch $\eta'(x),\eta'(y)$ only through (one of) their endpoints is by definition. The relation to preserved vertex bags is inherited from $\eta(e)$ and $\eta(f)$.
    \item All incidences between new edges are sound by construction, $\eta'(g')$ and $\eta'(g)$ do not touch other edge bags, and remaining new edge bags inherit their incidences from $\eta(e)$ and $\eta(f)$.
\end{enumerate}
\end{proof}}
\begin{restatable}{lemma}{lemshortmodelpreserved}\label{lem:short-model-preserved}
  Given a graph $G$ and a short $t$-model $\mH = (H,\eta,R)$ in $G$,
  the following operations, when applicable produce a short $t$-model:
  \begin{itemize}[itemsep=0px]
    \item Adding a shortest ear of distance at least $t+3$ on $\mH$.
    \item Adding a shortest ear between two distinct edges of $\mH$.
    \item Adding a shortest ear between an edge of $A$ and an edge of $B$
      for a partition $E(H) = A \uplus B$.
  \end{itemize}
  Furthermore, one can perform each of these operations (in particular, find the corresponding ear)
  in polynomial time, given $G$, $t$, and $\mH$.
\end{restatable}
\onlyarxiv{\begin{proof}
    Observe that the algorithmic claim is immediate: it suffices to iterate over 
    all choices $x,y$ of the endpoints of the sought ear and find a shortest path
    from $x$ to $y$ in $G-(N[V(\mH)] \setminus \{x,y\})$.

    Note that the second bullet follows from the third, as if $P$ is a shortest ear
    between two distinct edges of $\mH$, and $P$ is an ear between $e$ and $f$, $e \neq f$,
    then in particular $P$ is a shortest ear between $A \coloneqq \{e\}$ and $B \coloneqq E(H) \setminus \{e\}$. 
    In what follows, we prove the first and the last bullet at once, with minor details between these cases highlighted along the way.

    Let $\mH' = (H',\eta',R')$ be the $t$-model obtained by adding ear $P = p_1,\ldots,p_k$ between edges $e,f \in E(H)$ (that is, $N(p_1) \cap \mH \subseteq \eta(e)$ and $N(p_k) \cap \mH \subseteq \eta(f)$).
    If $P$ was taken as shortest across a partition $A \uplus B$, then assume w.l.o.g that $e \in A, f \in B$.
    Let $g' \in E(H')$ be the new edge with $\eta'(g') = P$, such that all other (new) edge bags are contained in $V(\mH)$.
    Note already that the first item of~\Cref{def:short-model} is satisfied. Indeed, new vertex bags each contain an endpoint of $P$, and the path on at most three vertices containing its neighbors, so their bag is of size at most four. Then, the containment in $R$ of these bags, along with length $t+3$ prefixes of each edge bag is given by definition.

    Assume for the sake of contradiction that $\mH'$ is not short.
    We will reach a contradiction by exhibiting a shorter ear than $P$ (between $A$ and $B$ in the second case).

    Let us first consider the case where for some $v \in N(\mH') \setminus N[R']$, $N(v) \cap \mH'$ is not contained in a three-vertex subpath of $H'$.
    Note that $N(v) \cap \mH'$ cannot be contained in $V(\mH)$, since $\mH$ was short and $R'$ contains prefixes at the ends of all new edges, so $N(v) \cap \mH'$ must intersect $P$.
    \begin{itemize}
      \item If $N(v) \cap \mH'$ is contained in $P$, then $v \notin N[\mH]$, and we let $x,y$ be the neighbors closest to $p_1,p_k$ respectively. Then, the path $P' = P[p_1,x] \cdot v \cdot P[y,p_k]$ is an ear of $\mH$, still between $p_1,p_k$, thus of distance at least $t+3$. It is of length strictly shorter than $P$, since $x,y$ are at distance at least $3$ on $P$, which contradicts $P$ being shortest (in particular across $A,B$).
      \item Otherwise, $N(v) \cap \mH'$ intersects $P = \eta'(g')$ as well as a single edge bag $\eta(h)$ for some $h \in E(H)$.
      In the case where $P$ was taken to be shortest across $(A,B)$, recall $e \in A$, and assume here without loss of generality that $h \in B$ (up to swapping $e$ and $f$).
      Take then $x \in P$ to be the neighbor of $v$ closest to $p_1$ on $P$, and let $P' = P[p_1,x] \cdot v$.
      Since $v \notin N[R']$, while $R'$ contains prefixes of length $t+3$ of all new edges,
      $P'$ is an ear of distance at least $t+3$.
      However, $P'$ is of length strictly shorter than $P$, since $P[x,p_k]$ contains a length $t+3$ prefix of $P$ (by definition of $R'$). This contradicts $P$ being the shortest.
    \end{itemize}

    We may now assume that, for any $v \in N(\mH') \setminus N[R']$, $N(v) \cap \mH'$ is contained in a three-vertex subpath of $\mH'$.
    Then, if $\mH'$ is not short, it is because of an ear $Q = q_1,\ldots,q_m$ of $\mH'$ that is of distance at least $t+3$ (with respect to $\mH'$) but 
    of length at most $t$ (i.e., $m \leq t + 1$).
    The proof follows very similar lines as the previous paragraph.
    Since $\mH$ was short, $Q$ cannot be an ear of $\mH$, so we may assume that for one of its endpoints, say $q_1$, $N(q_1) \cap \mH' \subseteq P$.
    Furthermore, the last paragraph ensures that both $N(q_1) \cap \mH'$ and $N(q_m) \cap \mH'$ are contained in a three-vertex subpath of (some edge bag) of $\mH'$.
    \begin{itemize}
      \item If $N(q_m) \cap \mH'$ is also contained in $P = \eta(g')$, let $x,y$ be their (respective) neighbors closest to $p_1,p_k$ respectively.
      Then, $P' = P[p_1,x] \cdot Q \cdot P[y,p_k]$ is still an ear of $\mH$ between $p_1,p_k$, thus of distance at least $t+3$.
      It is of length strictly shorter than $P$, since $x \cdot Q \cdot y$ is of length less than $t+3$, and $P[x,y]$ is of length at least $t+3$ by assumption on the distance of $Q$. This contradicts $P$ being the shortest.
      \item If $N(q_m) \cap \mH'$ is not contained in $P$, let $\eta(h)$ be the edge bag of $\mH$ containing it. Again, up to swapping $e$ and $f$, in the third bullet case, we may assume that $h \in B$.
      Take then $x$ to be the neighbor of $q_1$ closest to $p_1$ on $P$ (note that $x \notin R'$), and let $P' = P[p_1,x] \cdot Q$.
      As before, all neighbors of $q_m$ on $\mH'$ are at distance at least $t+3$ from $N(p_1) \cap \mH$.
      Then, $P'$ is an ear of $\mH$ between $p_1$ and $q_m$, thus of distance at least $t+3$. But $P'$ is of length strictly shorter than $P$, since $P[x,p_k]$ contains a length $t+3$ prefix of $P$ and $x \cdot Q$ is of length less than $t+2$. This contradicts $P$ being the shortest. \qedhere
    \end{itemize} 
\end{proof}
}

\begin{definition}[Ear-decomposition]\label{def:ear-decomposition-def}
    Given a short $t$-model $\mH_0$, a \emph{$t$-ear-decomposition} starting from $\mH_0$ is a sequence of models $\mH_0, \mH_1, \ldots , \mH_{\ell}$ such that for any $i \leq \ell-1$, $\mH_{i+1}$ is obtained from $\mH_i = (H_i,\eta_i,R_i)$ by adding an ear at distance at least $t+3$.

    We refer to indices $0 \leq i < \ell$ as \emph{steps}. A step is \emph{proper} if the ear added to 
    obtain $\mH_{i+1}$ from $\mH_i$ is between two distinct edges of $H$. 
\end{definition}

\subsection{Properties}
In our applications, the starting short $t$-model $\mH_0$ will be 
of constant size. A step increases the number of degree-$3$ vertices of $H$
by $2$, and does not change the number of vertices of degree at most $2$. 
A step can decrease the number of parallel edges by subdividing some of them.
Furthermore, a proper step
does not add parallel edges, while an improper step adds two parallel edges.
Finally, each ear addition adds at most $2 (3 + (t+3)) = 2t + 12$ vertices to $R$.
Thus, we have the following observation.
\begin{observation}\label{obs:ear-decomp-mostly-cubic}
  Let $\mH_0, \ldots, \mH_\ell$ be a $t$-ear-decomposition that contains $\ell'$ improper steps
  and $\mH_0 = (H_0,\eta_0,R_0)$.
  Then, $\mH_\ell$ contains at most $|V(H_0)|$ vertices of degree at most $2$, 
  at most $|E(H_0)| + 2\ell'$ parallel edges, and at least $2\ell$ vertices of degree $3$.
  Furthermore, $|R| \leq |R_0| + (2t+12)\ell$.
\end{observation}
Thus, if we start with a constant-size graph and generally avoid improper steps, we end with a graph
that is ``mostly'' simple and cubic. The lower bound on the number of degree-$3$ vertices
suffices to establish that $\mH_\ell$ contains many vertex-disjoint cycles thanks to \cref{thm:simonovits},
but vertices of lower degree and parallel edges count negatively if we want to extract thetas
and we need to control the number of them.

\medskip

A short $t$-model $\mH=(H,\eta,R)$ is \emph{maximal} if no ear at distance at least $t+3$ can be added to $\mH$,
and \emph{proper-maximal} if no ear between two distinct edges of $H$ can be added to $\mH$. 
\begin{lemma}\label{lem:max-model-no-inter-ear}
  Let $\mH = (H,\eta,R)$ be a proper-maximal short $t$-model. 
  Then, every connected component $D$ of $G-N[R]$ intersects at most one set $N[\eta(e)]$ for $e \in E(H)$
  and no set $N[\eta(v)]$ for $v \in V(H)$.
\end{lemma}
\begin{proof}
  The second claim is immediate as 
  for every $v \in V(H)$ we have $\eta(v) \subseteq R$ by the definition of a guarded model.
  For the first claim, assume that $D$ intersects $N[\eta(e)]$ for more than one $e \in E(H)$.
  Let $P \subseteq D$ be minimum such that $G[P]$ is connected and $P$ intersects $N[\eta(e)]$
  for at least two edges $e \in E(H)$. Then, as $\mH$ is short, $P$ is an induced path 
  between a vertex of $N[\eta(e_1)]$ and $N[\eta(e_2)]$ for two distinct edges $e_1,e_2 \in E(H)$.
  Hence, $P$ is an ear between two distinct edges of $H$, which is a contradiction.
\end{proof}

If we look at ears ending at the same edge of $H$, maximality has another important consequence,
but slightly more complicated to state.
An \emph{importance indicator} is a function that rates every induced subgraph of $G$
as ``important'' or ``not important'' such that if $A \subseteq B \subseteq V(G)$ and $G[A]$ is important, then so is $G[B]$.
In later applications, an induced subgraph will be important if it contains the object we hunt for: $C_{\geq t}$ or $\Xi_{\geq t}$.

\begin{lemma}\label{lem:split-one-component-from-model}
  Let $\mH = (H,\eta,R)$ be a maximal short $t$-model, let $D$ be a component of $G-N[R]$
  that contains a vertex of  $N[\eta(e)]$ for some $e \in E(H)$. 
  Assume that we are given some importance indicator in $G$.
  Then, there exists a subpath $P_D$ of $\eta(e)$ with at most $t+5$ vertices such that
  either 
  \begin{enumerate}[itemsep=0pt]
  \item no component $D'$ of $G[D \setminus N[P_D]]$ with $D' \cap N[\eta(e)] \neq \emptyset$ is important; or
  \item at least two components of $G[D \setminus N[P_D]]$ are important.
  \end{enumerate}
\end{lemma}
\begin{proof}
  If $D$ is not important, then we are done with $P_D = \emptyset$, so assume otherwise.

  Let $P = p_1,\ldots,p_m$ be the path of $G$ induced by $\eta(e)$.
  For each $i$, split $P$ into paths $L_i = p_1,...,p_i$, $M_i = p_i, \ldots p_{i+t+2}$, and $R_i = p_{i+t+2},...,p_m$.
  Each $M_i$ is of length $t+1$.

  We claim that for every $i$, there is no component $D'$ of $G-N[R]-N[M_i]$ that intersects both $N[L_i]$ and $N[R_i]$.
  Indeed, otherwise such a component contains a minimal path $P'$ connecting $N[L_i]$ and $N[R_i]$, 
  which, due to the length of $M_i$, would be an ear of distance at least $t+3$ connecting $e$ with itself. 
  This contradicts the maximality of $\mH$.

  Assume further that for no index $i$, setting $P_D \coloneqq M_i$ satisfies the conditions of the lemma.
  Then, for every $i$ there is a unique component $D_i$ of $G[D \setminus N[M_i]]$ that is both important
  and intersects $N[\eta(e)]$. Due to the claim established in the previous paragraph, $D_i$ intersects exactly
  one of $N[L_i]$ and $N[R_i]$.

  For $i=1$, $L_i \subseteq R$, hence $D_1$ must intersect $N[R_1]$. Symmetrically, for $i=m-(t+2)$, $R_i \subseteq R$
  and hence $D_{m-(t+2)}$ must intersect $N[L_{m-(t+2)}]$. We infer that there exists $1 \leq i < m-(t+2)$
  such that $D_i$ intersects $N[R_i]$ but $D_{i+1}$ intersects $N[L_{i+1}]$. 

  We claim that $P_D \coloneqq M_i \cup M_{i+1} = p_i, \ldots , p_{i+t+3}$ satisfies the conditions
  of the lemma. Assume otherwise: there exists exactly one connected component $D'$ of $G[D \setminus N[P_D]]$ 
  that is important and, furthermore, this component intersects $N[\eta(e)]$.
  Since $M_i \subseteq P_D$ and $D_i$ is the unique important component of $G[D \setminus N[M_i]]$, we have
  $D' \subseteq D_i$, and hence $D' \cap N[\eta(e)] \subseteq N[R_i]$. 
  Symmetrically, since $M_{i+1} \subseteq P_D$ and $D_{i+1}$ 
  is the unique important component of $G[D \setminus N[M_{i+1}]]$, we have
  $D' \subseteq D_{i+1}$, and hence $D' \cap N[\eta(e)] \subseteq N[L_{i+1}]$. 
  However $N[L_{i+1}] \cap N[R_i] \subseteq N[R]$ due to $t \geq 1$ and $\mH$ being short.
  This is the desired contradiction.
\end{proof}

	\section{Algorithmic induced Erd\H{o}s--P\'osa property for long holes}\label{sec:cycles}
  	
\EPCyclesAlgo*
\begin{proof}
    Let $\lsim \geq 1$ be such that 
    $\fsim(k) \leq \lsim k \log(2k)$ for every nonnegative integer $k$,
    where $\fsim(k)$ is the function of \cref{thm:simonovits}.
    In this proof, we use the natural 
    convention $0 \cdot \log 0 = 0$; note that this keeps $x \mapsto x \log (2x)$ continuous 
    on $[0,+\infty)$. The constant $2$ in the logarithm's argument makes the bound correct for $k=1$.

    The algorithm will be recursive; 
    As the second output, it will return $X$ of size bounded by $\lambda kt \log (2k)$ for $\lambda \coloneqq 896(\lsim+1)$.
    We can assume that the input graph $G$ is connected. 
    Otherwise, we call the algorithm for each component independently and return the union of objects found for the components. A standard convexity argument shows that $\sum_{i=1}^p \lambda  k_it \log (2k_i) \leq \lambda kt \log (2k)$ for $k = \sum_{i=1}^p k_i$, so the bound on $|X|$ is preserved.

    Let $G$ and $t$ be given on input. We apply \cref{lem:init-cycle}.
    If no $C_{\geq t}$ is found, we report $k=0$, an empty cycle packing, and $X = \emptyset$.
    If a hole $C$ of length exactly $t$ is found, we recurse on $G-N[C]$, obtaining $k'$, a cycle packing
    $\mathcal{C}'$, and a set $X'$ of size at most $\lambda k' t \log(2k')$.
    We return $k \coloneqq k' + 1$,
    cycle packing $\mathcal{C}' \cup \{C\}$, and $X \coloneqq X' \cup V(C)$; as $|V(C)| \leq t$ and $\lambda \geq 1$,
    clearly $|X| \leq \lambda kt \log (2k)$.
    
    Assume then that a short $t$-model $\mH_0 = (H_0,\eta_0,R_0)$ is returned with $|R_0| \leq 2t+7$
    and $|V(H_0)| = |E(H_0)| = 1$.
    While possible, we add a shortest ear of distance at least $t+3$ to $\mH_0$, obtaining 
    an ear decomposition $\mH_0, \mH_1, \ldots, \mH_\ell$. 
    Note that $\mH_\ell$ is a maximal short $t$-model and, by \cref{obs:ear-decomp-mostly-cubic},
    contains at least $2\ell$ vertices of degree $3$.
    For brevity, denote $\mH = \mH_\ell = (H,\eta,R)$; we have $|R| \leq (2t+7) + \ell(2t+12) < (\ell+1)(2t+12)$.
    
    By \Cref{obs:model-subdivision}, $\mH$ is an induced minor model of a $t$-subdivision of $H$.
    Let $k_\mH^\circ$ be the maximum integer such that $\lsim k_\mH^\circ \log(2k_\mH^\circ) \leq 2\ell$, or
    $k_\mH^\circ = 1$ if $\ell = 0$. We have $k_\mH^\circ \geq 1$ and,
    by \cref{thm:simonovits} (and, for $\ell=0$, by the fact that $H_0$ is a cycle),
    we can find a packing $\mathcal{C}_H$ of $k_\mH^\circ$ vertex-disjoint cycles in $H$.
    By the choice of $k_\mH^\circ$ and the fact that $k_\mH^\circ \geq 1$, 
    \[ 2\ell < \lsim(k_\mH^\circ+1)\log(2(k_\mH^\circ+1)) \leq 2\lsim k_\mH^\circ \log(4k_\mH^\circ) \leq 4\lsim k_\mH^\circ \log(2k_\mH^\circ). \]
    
    We would like to recurse on the connected components of $G-N[R]$ and use 
    in the packing the cycles of $\mathcal{C}_H$ (projected via \cref{cor:cycles-from-model}),
    but these may touch each other, as the components of $G-N[R]$ may intersect
    sets $N[\eta(e)]$ for $e \in E(H)$.
    Due to \cref{lem:max-model-no-inter-ear}, 
    every component $D$ of $G-N[R]$ intersects at most one set $N[\eta(e)]$. 
    We would like to sacrifice cycles of $\mathcal{C}_H$ that use such sets $e$ where $D$ still
    contains a $C_{\geq t}$.
    This decreases the size of the packing $\mathcal{C}_H$, but this will be offset by the fact that
    there will be multiple components $D$ that contain a $C_{\geq t}$ and we will gain by the convexity
    of the function $k \mapsto k \log (2k)$. 

    However, this charging does not work in the corner case where there is a unique component $D$
    of $G-N[R]$ that contains a $C_{\geq t}$ and $D$ intersects $N[\eta(e)]$ for some $e \in E(H)$.
    Note that we can discover that this is the case with \cref{lem:init-cycle}.
    In this case, we use \cref{lem:split-one-component-from-model}
    (for an induced subgraph being important if it contains a $C_{\geq t}$):
     for some set $P_D$ of at most $t+5$
    consecutive vertices of $\eta(e)$, in $G-N[R \cup P_D]$ this corner case no longer appears. 
    As we can detect the corner case in $n^{\Oh(t)}$ time with \cref{lem:init-cycle}, we can check all
    $\Oh(nt)$ candidates for $P_D$. 

    In the end, we obtain a $R' \coloneqq R \cup P_D$ of size less than $(\ell+2)(2t+12)$ 
    such that $\mH' = (H,\eta,R')$ is a maximal short $t$-model.
    Let $D_1,\ldots,D_p$ be the connected components of $G-N[R']$ that contain a $C_{\geq t}$.
    By the exclusion of the corner case we have either $p \geq 2$
    or $\bigcup_{i=1}^p D_i \cap N[V(\mH)] = \emptyset$.
    Let $E' \subseteq E(H)$ be the set of edges $e$ such that $N[\eta(e)] \cap D_i \neq \emptyset$ for
    some $1 \leq i \leq p$. 
    Let $\mathcal{C}_H' \subseteq \mathcal{C}_H$ be the set of those cycles of $\mathcal{C}_H$
    that do not use any edge of $E'$. As $|E'| \leq p$, we have $|\mathcal{C}_H'| \geq k_\mH^\circ - p$.
    By \cref{cor:cycles-from-model}, $\mathcal{C}_H'$ projects
    to a packing $\mathcal{C}_\mH$ of $|\mathcal{C}_H'|$ anti-adjacent $C_{\geq t}$s in $G[V(\mH)]$ that do not use 
    vertices of $\eta(e)$ for $e \in E'$. That is, the cycles of $\mathcal{C}_\mH$ are anti-adjacent
    to the components $D_i$.

    For every $i=1,2,\ldots,p$, we recurse on $G_i \coloneqq G[D_i]$, obtaining $k_i$, $\mathcal{C}_i$,
    and $X_i$. We return 
    \[
    k \coloneqq |\mathcal{C}_\mH| + \sum_{i=1}^p |\mathcal{C}_i|, \quad 
    \mathcal{C} \coloneqq \mathcal{C}_\mH \cup \bigcup_{i=1}^p \mathcal{C}_i, \quad 
    X \coloneqq R' \cup \bigcup_{i=1}^p X_i.
    \]
    It is immediate that $\mathcal{C}$ is a packing of anti-adjacent $C_{\geq t}$s of size $k$
    and that $G-N[X]$ does not contain any $C_{\geq t}$. It remains to bound the size of $X$.

    Recall $\lambda \coloneqq 896(\lsim+1) = 32 \cdot 28 \cdot (\lsim+1)$ and $2\ell \leq 4\lsim k_\mH^\circ \log(2k_\mH^\circ)$. Thus, as $k_\mH^\circ,t \geq 1$,
    \begin{equation}\label{eq:cycles:R}
        |R'| \leq (\ell+2)(2t+12) \leq \left(2\lsim k_\mH^\circ \log (2k_\mH^\circ) + 2\right)(2t+12) \leq \frac{1}{32} \lambda tk_\mH^\circ \log (2k_\mH^\circ). 
    \end{equation}

    Consider first the case $|\mathcal{C}_\mH| \geq k_\mH^\circ/2$. This implies $\sum_{i=1}^p k_i \leq k-k_\mH^\circ/2$.
    Hence,
    \begin{align*}
    |X| &= |R'| + \sum_{i=1}^p |X_i| \leq \frac{1}{32}\lambda tk_\mH^\circ \log(2k_\mH^\circ) + \sum_{i=1}^p \lambda t k_i \log(2k_i)\\
    &\leq t \log(2k) \cdot \left(\frac{1}{32} \lambda k_\mH^\circ + \lambda (k - k_\mH^\circ/2)\right) \leq \lambda tk \log(2k).
    \end{align*}

    Consider now the case $|\mathcal{C}_\mH| < k_\mH^\circ/2$, so in particular $p > k_\mH^\circ/2$ as $|\mathcal{C}_\mH| \geq k_\mH^\circ - p$.
    Without loss of generality, assume $k_1 \geq k_2 \geq \ldots \geq k_p$.
    This implies that for every $1 \leq i \leq p$, we have $k_i \leq k/i$, hence
    $\log(2k_i) \leq \log(2k) - \log(i)$.
    We must have $p \geq 2$ here, as if $p \leq 1$ we have $\bigcup_{i=1}^p D_i \cap N[V(\mH)] = \emptyset$ so $E' = \emptyset$
    and $|\mathcal{C}_\mH| = k_\mH^\circ$.
    We infer:
    \begin{align*}
    |X| &= |R'| + \sum_{i=1}^p |X_i| \leq \frac{1}{32}\lambda tk_\mH^\circ \log(2k_\mH^\circ) + \sum_{i=1}^p \lambda t k_i \log(2k_i)\\
    &\leq \frac{1}{16}\lambda tp \log(4p) + \lambda t \sum_{i=1}^p k_i \left(\log(2k) - \log(i)\right)\\
    &\leq \frac{1}{8}\lambda tp \log(2p) + \lambda tk \log(2k) - \lambda t \sum_{i=2}^p k_i \log(i).
    \end{align*}
    As $p \geq 2$, we have
    \[ \sum_{i=2}^p k_i \log(i) \geq \sum_{i=\lceil \frac{p+1}{2} \rceil}^p \log\left(\left\lceil \frac{p+1}{2} \right\rceil\right) \geq \frac{p}{2} \cdot \frac{1}{4}\log(2p) = \frac{1}{8} p \log(2p).\]
    Thus, we have $|X| \leq \lambda tk \log(2k)$, as desired. 
\end{proof}

\paragraph{Deducing \cref{thm:EP-cycles-algo} from \cref{thm:EP-3LPC-algo}.}
We conclude this section by showing how \cref{thm:EP-cycles-algo} can be deduced
from \cref{thm:EP-3LPC-algo}. Given $G$ and $t$, invoke
\cref{thm:EP-3LPC-algo}, obtaining $k_\Xi$, a packing $\mathcal{C}_\Xi$
of $k_\Xi$ $t$-long 3PCs and a set $X_\Xi$ of size $\Oh(tk_\Xi \log(k_\Xi))$
such that $G-N[X_\Xi]$ has no $\Xi_{\geq t}$. As every $\Xi_{\geq t}$ contains
a $C_{\geq t}$, extract from $\mathcal{C}_\Xi$ a packing $\mathcal{C}^1$
of $k_{\Xi}$ anti-adjacent $C_{\geq t}$s.

We focus on $G' \coloneqq G-N[X_\Xi]$ and we recurse on $G'$ and $t$,
but this time knowing that $G'$ has no $\Xi_{\geq t}$. 
If we obtain an output $k'$, $\mathcal{C}'$, and $X'$, we can return 
the larger of the packings $\mathcal{C}'$ and $\mathcal{C}^1$ (with its size)
and $X_\Xi \cup X'$.
Hence, it suffices to prove the following lemma.
\begin{lemma}\label{lem:ep-cycles-no-xi}
    Given an integer $t$ and a graph $G$ that does not contain a $\Xi_{\geq t}$,
    one can in time $|V(G)|^{\Oh(t)}$ compute an integer $k$,
    a family of $k$ vertex-disjoint anti-adjacent induced cycles,
    each of length at least $t$, and a set $X \subseteq V(G)$
    of size $\Oh(kt)$ such that $G-N[X]$ does not contain
    any induced cycle of length at least $t$.
\end{lemma}
\begin{proof}
We proceed as in the proof of \cref{thm:EP-cycles-algo} by
invoking \cref{lem:init-cycle}. The critical observation is that 
$\mH_0 = (H_0, \eta_0, R_0)$ will already be maximal, as adding any ear to it would create a
short $t$-model that contains a theta, so the underlying induced subgraph of $G'$
would contain $\Xi_{\geq t}$. 
After a possible application of \cref{lem:split-one-component-from-model}
if there is exactly one component of $G-N[R_0]$ containing a $C_{\geq t}$ and this
component intersects $N_{G}[V(\mH)]$, we obtain a set $R$ of size $\Oh(t)$
such that either (a) at least two components of $G-N[R]$ contain a $C_{\geq t}$, or
(b) at most one component of $G-N[R]$ contains a $C_{\geq t}$, but this component,
if it exists, is disjoint with $N[V(\mH)]$. 

In both cases, we add $R$ to the returned deletion set.
In case (a), we recurse on the components of $G-N[R]$ that contain a $C_{\geq t}$
and return the union of their outputs.
In case (b), we recurse on the unique component of $G-N[R]$ that contains a $C_{\geq t}$
(if exists) and add the $C_{\geq t}$ hidden in $G[V(\mH)]$ to the returned packing.

Let $k$ be the size of the returned cycle packing.
We argue that the returned deletion set $X$ is of size $\Oh(k t)$.
Clearly, every recursion with case (b) adds a cycle to the output packing,
thus there are at most $k$ such recursive calls.
Note that we only recurse on components containing a $C_{\geq t}$ and 
use the returned cycle packing in the output one. 
Thus, the recursion tree has at most $k$ leaves. Consequently,
there are at most $k-1$ recursive calls with case (a). 
Since at each case we have $|R| = \Oh(t)$, the bound $|X| = \Oh(k t)$ follows.
\end{proof}

	\section{Preliminaries}\label{sec:prelims}
	\paragraph{Basic graph notation.}
Given a (multi)graph $G$, we denote the set of vertices of $G$ as $V(G)$ and the set of edges as $E(G)$.
We write $G[A]$ for the subgraph induced by $A \subseteq V(G)$; $G - A$ is a shorthand for $G[V (G) \setminus A]$.
If it does not lead to confusion, we do not distinguish between sets of vertices and subgraphs induced by them. In particular, we identify components and their vertex sets.

An edge is incident to both its endpoints, and a pair of edges is incident if they share an endpoint.
We say that subsets $A,B \subseteq V(G)$ \emph{touch} if they either intersect or are adjacent.
Two sets are \emph{anti-adjacent} if they do not touch.

The \emph{degree} of a vertex $v$ is the number of edges incident to a vertex $v$, where a loop contributes two to the degree of its only incident vertex.
A graph $H$ is called \emph{subcubic} if all its vertices have degree at most $3$. A graph $H$ is called \emph{cubic} if all its vertices have degree $3$.
For a subcubic $G$ we denote by $\dlttwo(G)$ and $\dthree(G)$ the number of vertices of degree at most 2 and exactly $3$, respectively.

The set of neighbors of a vertex $u$ is denoted as $N_G(u)$ (so called \emph{open neighborhood}) and the set $N_G(u)\cup \{u\}$ is denoted as $N_G[v]$ (so called \emph{closed neighborhood}).
For a set of vertices $X$ of $G$, its open neighborhood is $N_G(X)=\bigcup_{v\in X} N_G(v) \setminus X$ and its closed neighborhood is $N_G[X]=\bigcup_{v\in X} N_G(v) \cup X$.

\paragraph{Packings.}

Let $\cH$ be a family of graphs.
An \emph{induced $\cH$-packing} in a graph $G$ is a family of pairwise disjoint and anti-adjacent subsets of $V(G)$, each of which induces a graph from $\cH$.
The size of a largest induced $\cH$-packing in $G$, i.e., the number of sets in the packing, is denoted as $\pack_\cH(G)$.

\paragraph{Tripods.} In a few places, we will need a simple well-known fact that a minimal
induced subgraph connecting three given vertices is a tree or a line graph of a tree. 
We formalize it with the following notion.

\begin{definition}[tripod]
A \emph{tripod} in a graph $G$ is a set of three induced paths that are disjoint and anti-adjacent except for the following:
each path has a designated one endpoint as a start and the second endpoint as an end and we require that the starts are either all equal or all pairwise
adjacent (i.e., form a triangle). The ends of the three paths are called the ends of the tripod.
We explicitly allow the paths in a tripod to be of length $0$.
\end{definition}

It is straightforward to verify (see e.g. \cite[Claim 2]{DBLP:conf/stoc/GartlandLPPR21}) the following:
\begin{lemma}\label{lem:make-tripod}
	Let $G$ be a graph and $x,y,z \in V(G)$ be three distinct vertices
	in the same connected component of $G$.
	Then, any inclusion-wise minimal induced subgraph of $G$ that is connected and contains
	$x$, $y$, and $z$, is a tripod with ends $x$, $y$, $z$.
	Furthermore, such an induced subgraph can be found in polynomial time.
\end{lemma}

\subsection{Long thetas and  long three-path-configurations}\label{sec:theta-xi-section}

We now show \cref{lem:theta-xi}
\obsThetaXi*

We will need to introduce some notation and terms.
Let $G$ and $H$ be graphs, where $H$ has $k$ vertices.
An induced minor model of $H$ in $G$ is a family $\cX = \{ X_v \}_{v \in V(H)}$ of pairwise disjoint connected subsets of $V(G)$,
such that $X_u$ and $X_v$ touch if and only if $uv \in E(H)$.

The following property of induced minor models is well-known, see e.g.~\cite[Lemma 2.1~1.]{DBLP:journals/jcss/BousquetDDHMPT26}

\begin{lemma}\label{lem:minor-models-degree-2}
Let $G$ and $H$ be graphs.
Given an induced minor model of $H$ in $G$,
in polynomial time one can find an induced minor model $\cX = \{ X_v \}_{v \in V(H)}$ of $H$ in $G$ such that for every vertex $u \in V(H)$ that has degree at most 2 and belongs to a component of $H$ that is not a cycle, we have $|X_u| = 1$.
\end{lemma}

Now we are ready to prove \cref{lem:theta-xi}.
\begin{proof}[Proof of \cref{lem:theta-xi}]
    To prove the first item, it is enough to consider induced minor models of $\Theta_{t+2}$, as $\Theta_{\geq t+2}$ contains $\Theta_{t+2}$ as an induced minor.
    Denote the degree-3 vertices of $\Theta_{t+2}$ as $y$ and $z$,
    and the degree-2 vertices of each of the three paths by $a_1,\ldots,a_{t+2}$, $b_1,\ldots,b_{t+2}$, and $c_1,\ldots,c_{t+2}$, respectively, so that $a_1,b_1,c_1$ are neighbors of $y$, and $a_{t+2},b_{t+2},c_{t+2}$ are neighbors of $z$.

    Let $G$ be a graph that contains $\Theta_{t+2}$ as an induced minor.
    By \cref{lem:minor-models-degree-2}, consider an induced minor model $\cX = \{ X_v \}_{v \in V(\Theta_{t+2})}$ of $\Theta_{t+2}$ in $G$ where $|X_{a_i}| = |X_{b_i}| = |X_{c_i}| = 1$ for every $i \in [t+2]$.

    Note that each of sets $\bigcup_{i=2}^{t+1} X_{a_i}$, $\bigcup_{i=2}^{t+1} X_{b_i}$, and $\bigcup_{i=2}^{t+1} X_{c_i}$ induces a path on $t$ vertices, and these paths are pairwise disjoint and anti-adjacent.

    Let $a',b',c'$ be the unique vertices in $X_{a_1}$, $X_{b_1}$, and $X_{c_1}$, respectively.
    Note that the set $X_y \cup X_{a_1} \cup X_{b_1} \cup X_{c_1}$ is connected. 
    By Lemma~\ref{lem:make-tripod}, an inclusion-wise minimal connected induced subgraph of $G[X_y \cup X_{a_1} \cup X_{b_1} \cup X_{c_1}]$ that contains $a'$, $b'$, and $c'$ is a tripod with ends $a',b',c'$.

    Combining this tripod with an analogous one obtained for $X_z \cup X_{a_{t+2}} \cup X_{b_{t+2}} \cup X_{c_{t+2}}$ and the three paths between them, we obtain a $\Xi_{\geq t}$ in $G$.
\medskip

    The second item is straightforward: $\Theta_{\geq t}$ is obtained from $\Xi_{\geq t}$ by contracting the triangles consisting of the endpoints of paths, if such exist.
\end{proof}

Let us emphasize that the slack of 2 in the first item of \cref{lem:theta-xi} is necessary, as there are graphs that contain $\Theta_{t}$ as an induced minor but do not contain $\Xi_{\geq t-1}$, see \cref{fig:turtle}.

\begin{figure}
    \centering
\includegraphics[scale=1.2]{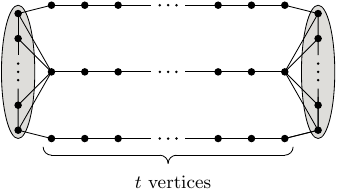}
\caption{An induced minor model of $\Theta_{t}$ that does not contain $\Xi_{\geq t-1}$.
Shaded areas indicate the sets of vertices corresponding to the degree-3 vertices of $\Theta_{t}$, and other vertices are represented by single vertices.
}\label{fig:turtle}
\end{figure}

	\section{Obtaining long 3PCs from (sub)cubic models}\label{sec:pack-in-multigraph}
	\subsection{Thetas in models give long 3PCs}

We start with the following analog of \cref{cor:cycles-from-model}.
\begin{lemma}\label{lem:xis-from-model}
  If $\mH$ is a $t$-model of $H$ in $G$ and $H$ admits a packing of $k$
  vertex-disjoint minor models of a theta, then $G[V(\mH)]$ admits a packing of $k$
  anti-adjacent $\Xi_{\geq t}$s.

  Furthermore, given the model and the packing of theta models, one can produce 
  the corresponding packing of $\Xi_{\geq t}$s in polynomial time.
\end{lemma}
\begin{proof}
	It suffices to prove that, given a minor model of a theta in $H$, one can
	produce in polynomial time a $\Xi_{\geq t}$ in the subgraph of $G$ induced
	by the bags of the vertices and edges used in the said minor model. 

	Since theta is a subcubic graph, we can change the input minor model into a topological
	minor model of a theta in $H$: we have two distinct vertices $a,b \in V(H)$
	connected by three paths $P_1,P_2,P_3$ that are vertex-disjoint except for the endpoints.

	Fix $i \in [3]$. By the definition of a $t$-model, the union of bags of edges
	and internal vertices of $P_i$ induce a connected subgraph of $G$ that both 
	contain a vertex of $\eta(a)$ and a vertex of $\eta(b)$. Let $Q_i$ be a shortest
	path in this induced subgraph between $\eta(a)$ and $\eta(b)$. 
	By the definition of a $t$-model, $Q_i$ contains all paths of $\eta(e)$
	for $e \in E(P_i)$ and paths within $\eta(v)$ for every internal vertex $v$ of $P_i$.
	In particular, $Q_i$ contains at least $t$ internal vertices, as $P_i$ contains at least one edge.

	Let $Q_i'$ be the subpath of $Q_i$ with the endpoints omitted; $|V(Q_i')| \geq t$.
	Let $q_{i}^a$ and $q_i^b$ be the endpoints of $Q_i'$ that are neighbors
	of $\eta(a)$ and $\eta(b)$, respectively. 
	By \cref{lem:make-tripod}, for $\gamma \in \{a,b\}$, 
	we can find a tripod $T^\gamma$ connecting $q_1^\gamma$, $q_2^\gamma$, $q_3^\gamma$
	in $G[\eta(\gamma) \cup \{q_1^\gamma,q_2^\gamma,q_3^\gamma\}]$ in polynomial time.
	It is straightforward to check that the union of $T^a$, $T^b$, $Q_1'$, $Q_2'$, and $Q_3'$
	form the desired $\Xi_{\geq t}$.
\end{proof}

\subsection{Packing thetas in a subcubic model}

We will need the following analog of \cref{thm:simonovits} for thetas. 

\begin{restatable}{theorem}{packthetasinH}\label{thm:thetas-in-cubic}
There exists a function $\ftheta(k) = \Oh(k \log k)$ with the following property.
For every $k \geq 1$, 
if a subcubic multigraph without loops satisfies $\dthree(G) \geq \ftheta(k) + 8\dlttwo(G) + 18n_{||}$,
where $n_{||}$ is the number of parallel edges in $G$,
then $G$ contains $k$-vertex-disjoint thetas (as a subgraph).
Furthermore, these thetas can be found in time polynomial in $|V(G)|$.
\end{restatable}

Note that in \cref{thm:thetas-in-cubic} we penalize for parallel edges.
To see that this is necessary, consider a subcubic multigraph obtained from a path by duplicating every second edge. This multigraph has a constant number of vertices of degree less than $3$, many vertices of degree 3, but it does not contain any theta as a subgraph.

\medskip

A multigraph is a \emph{cactus} if its every block (inclusion-wise maximal biconnected subgraph) is either a cycle or an edge.
Equivalently, a cactus is a multigraph that does not contain two cycles that share an edge.
This latter characterization can be rephrased in the following way that will be very useful to us.

\begin{proposition}
A multigraph is a cactus if and only if it does not contain any theta as a subgraph.
\end{proposition}

To show \cref{thm:thetas-in-cubic}, we will use the \EP-type theorem for vertex-disjoint thetas.

\begin{theorem}[Chatzidimitriou, Raymond, Sau, Thilikos~\cite{DBLP:journals/algorithmica/Chatzidimitriou18}]\label{thm:EP-thetas-subgraph}
There exists a function $f_{\textrm{CRST}}(k) = \Oh(k \log k)$ such that the following holds.
Given a graph $G$ and an integer $k$, in polynomial time we can find either $k$ vertex-disjoint thetas as a subgraph of $G$ or a set $X \subseteq V(G)$ of size at most $f_{\textrm{CRST}}(k)$ such that $G-X$ does not contain any theta as a subgraph.
\end{theorem}

We will also need the following lemma.

\begin{lemma}\label{lem:hittingsetforthetas}
Let $G$ be a simple subcubic graph and let $X \subseteq V(G)$ be such that $G - X$ does not contain any theta as a subgraph.
Then $|V(G)| \leq 10|X| + 9\dlttwo(G)$.
\end{lemma}
\begin{proof}
Let $T$ denote the graph $G - X$. Note that this is a subcubic cactus: a graph whose every block is an edge or a cycle.

\begin{claim}\label{clm:edgesincactus}
Every simple $n$-vertex subcubic cactus has at most $\frac{4}{3}n$ edges.
\end{claim}
\begin{proofofclaim}
Let $H$ be an $n$-vertex subcubic cactus with largest possible number of edges.
Note that this implies that $H$ is connected.
Indeed, it is easy to observe that every cactus has a vertex of degree at most 2, and adding an edge between two such vertices in different components creates a new subcubic cactus with more edges.

The proof is by induction on the number of blocks.
If $H$ has one block, then the number of edges is at most $n$ and we are done.
Also, if $H$ is a tree, then clearly its number of edges is $n-1 \leq \frac{4}{3}n$.
Thus, suppose that $H$ has at least two blocks, and at least one of them is a cycle.
Pick an arbitrary block $R$ that is a cycle and consider it as a root.
Note that $H$ has tree-like structure: $R$ is the root, all blocks that intersect $R$ are children of $R$, and so on.
Furthermore, since $H$ is subcubic, the parent of a block that is a cycle is always an edge.

Let $B$ be a block that corresponds to a leaf in the tree discussed above.
We consider two cases.

If $B$ is a single edge, denote its vertices by $vv'$, where $v$ is the common vertex of $B$ and its parent block.
The graph $H - v'$ has $n-1$ vertices and one block fewer than $H$, so by induction it has at most $\frac{4}{3}(n-1)$ edges.
Consequently, $H$ has at most $\frac{4}{3}(n-1)+1 \leq \frac{4}{3}n$ edges.

So, assume that $B$ is a cycle on $k \geq 3$ vertices.
Let its parent block be $B'$ and recall that $B'$ is an edge. In particular, $B'$ is not the root block.
Let us consider the graph $H - B$, i.e., we remove all vertices from $B$. Note that this also removes the block $B'$.
Thus, $H-B$ has $n-k$ vertices and two blocks fewer than $H$, so, by induction, its number of edges is at most $\frac{4}{3}(n-k)$.
Consequently, the number of edges of $H$ is at most
\[
		\frac{4}{3}(n-k) + k + 1 \leq \frac{4}{3}n,
\]
since $k \geq 3$.
\end{proofofclaim}

Let $x=|X|$ and $n = |V(T)|$, i.e., $G$ has $n+x$ vertices.
We denote $\dlttwo \coloneqq \dlttwo(G)$ and $\dthree \coloneqq \dthree(G)$ for brevity.
Since $G$ is subcubic, its number of edges is at least
$\frac{3}{2} \dthree = \frac{3}{2}(n+x - \dlttwo)$.
By \cref{clm:edgesincactus}, $T$ has at most $\frac{4}{3}n$ edges.
As every vertex of $X$ is incident to at most three edges,
removing $X$ removed at most $3x$ edges.
Hence, 
\[ \frac{3}{2} (n + x - \dlttwo) \leq |E(G)| \leq |E(T)| + 3x \leq \frac{4}{3}n + 3x. \]
After rearranging, 
\[ \frac{1}{6}n \leq \frac{3}{2}x + \frac{3}{2} \dlttwo. \]
Hence,
\[ |V(G)| = n+x \leq 10x + 9\dlttwo.\]
This completes the proof of the lemma.
\end{proof}

Now we are ready to prove \cref{thm:thetas-in-cubic}.

\begin{proof}[Proof of \cref{thm:thetas-in-cubic}]
We start with observing that is enough
to prove the statement for the case $n_{||} = 0$. Indeed, deleting from $G$ a single edge
that has a parallel edge decreases $n_{||}$ by at least one,
while decreases $\dthree(G)$ by at most two and increases $\dlttwo(G)$ by at most two.
Thus, in what follows we assume $n_{||} = 0$; with the assumption
that $G$ does not contain a loop, $G$ is in fact a simple graph.

Let $f_\Theta(k) = 10 f_{\textrm{CRST}}(k) +1 = \Oh(k \log k)$, where $f_{\textrm{CRST}}$ is the function from \cref{thm:EP-thetas-subgraph} and consider a simple subcubic graph $G$
with $\dthree(G) \geq f_\Theta(k) + 8\dlttwo(G)$.
We invoke the algorithm from \cref{thm:EP-thetas-subgraph} for $G$ and $k$.
If the call returns $k$ vertex-disjoint thetas, we are done.
So suppose otherwise, i.e., there is $X \subseteq V(G)$ of size at most $f_{\textrm{CRST}}(k)$ such that $G - X$ has no theta as a subgraph.
Applying \cref{lem:hittingsetforthetas}, we conclude that 
\[ |V(G)| \leq 10|X| + 9\dlttwo(G) \leq 10f_{\textrm{CRST}}(k) + 9\dlttwo(G). \]
Hence,
\[ \dthree(G) = |V(G)| - \dlttwo(G) \leq 10f_{\textrm{CRST}}(k) + 8\dlttwo(G) < f_\Theta(k) + 8\dlttwo(G). \]
This is a contradiction.
\end{proof}

	\section{Algorithmic induced Erd\H{o}s--P\'osa property for long 3PCs}\label{sec:L3PC}
	
In this section, we prove the main technical result of the paper,
namely \cref{thm:EP-3LPC-algo}.

We start by showing that a $t$-long 3PC always gives raise
to a short $t$-model $\mH_0$ of $\ThetaZero$ (\cref{lem:detect-L3PC}).
We would like to proceed as in the proof of \cref{thm:EP-cycles-algo}:
add ears to the model as long as possible, obtaining a maximal model $\mH=(H,\eta,R)$,
and then argue that we can recurse on the connected components of $G-N[R]$.
\cref{thm:thetas-in-cubic} gives us a large packing of thetas in $H$ and \cref{lem:theta-xi}
translates it into a packing of anti-adjacent $\Xi_{\geq t}$s. 
However, this strategy hits two problems. 

First, adding an ear with the endpoints on the same edge does not give progress:
it produces parallel edges, which do not necessarily lead to more theta minors
in $H$. Hence, we need to allow only ears connecting two distinct edges of the model.

Second, the corner case we handled with \cref{lem:split-one-component-from-model}
is more complicated here, partially because we do not allow adding ears with endpoints
on the same edge. Thus, we prove a much more elaborate version of \cref{lem:split-one-component-from-model}
that grows the model until the following property is achieved: deletion of a single edge
from it does not destroy all theta minor models. 
This is enough to escape the corner case and proceed with the recursion as in 
the proof of \cref{thm:EP-cycles-algo}.

\subsection{Long 3PCs and models of $\ThetaZero$}
In this section we show how to detect a $\Xi_{\geq t}$ in a graph
and how to produce a short $t$-model of $\ThetaZero$ from it.
We will need the following slight generalization of a tripod.
\begin{definition}[partial $t$-long 3-path configuration]\label{def:partial_tL3PC}
	Let $t$ be a positive integer and let $G$ be a graph.
	A \emph{partial $t$-long $3$-path configuration} is a set of three induced paths $P_1$, $P_2$, $P_3$,
	in $G$, each of length exactly $t+1$ (i.e., with $t$ internal vertices),
	that are pairwise disjoint and anti-adjacent, except for the following: 
	each path $P_i$ has a designated one endpoint $s_i$ as a start and the second endpoint $t_i$ as an end and we require that the starts are either all equal or all pairwise
    adjacent (i.e., form a triangle) and we allow each pair of ends to be equal or adjacent.

	The \emph{shell} of the partial $t$-long 3-path configuration $\mathcal{P}= \{P_1,P_2,P_3\}$
	is defined as 
	 \[ \mathrm{Shell}(\mathcal{P}) = N[V(P_1) \cup V(P_2) \cup V(P_3) \setminus \{t_1,t_2,t_3\}] \setminus \{t_1,t_2,t_3\}. \]
	We say that $\mathcal{P}$ is \emph{promising} if $t_1$, $t_2$, and $t_3$ lie in the same connected
	component of $G-\mathrm{Shell}(\mathcal{P})$.
\end{definition}
Clearly, each tripod with paths of length exactly $t+1$ is a partial $t$-long 3-path configuration.
We have the following immediate observation.
\begin{lemma}\label{lem:xi-to-partial-l3pc}
	Let $P_1$, $P_2$, $P_3$ be a $\Xi_{\geq t}$ in a graph $G$.
	Then, the triple consisting of the three prefixes of the paths $P_1$, $P_2$, $P_3$
    of length exactly $t+1$ is a promising partial $t$-long 3-path configuration.
\end{lemma}
\begin{proof}
  The claim that these prefixes form a $t$-long $3$-path configuration is immediate.
  To see that it is promising, note that if 
  $P_1'$, $P_2'$, and $P_3'$ are the said prefixes, $t_i$ be the end of $P_i'$,
  and $Q_i$ is the other part of $P_i$ starting from $t_i$ to the end, then
  $Q_1 \cup Q_2 \cup Q_3$ is disjoint with $\mathrm{Shell}(\{P_1', P_2', P_3'\})$,
  connected, and contains $t_1$, $t_2$, and $t_3$.
\end{proof}
In the other direction, we observe that a promising partial $t$-long $3$-path configuration
extends to a $\Xi_{\geq t}$ or, even more, to a short $t$-model of $\ThetaZero$. 

\begin{lemma}\label{lem:extend_tripod_to_theta}
	Let $G$ be a graph, let $t$ be a positive integer, and let $\mathcal{P} = \{P_1, P_2, P_3\}$
	be a promising partial $t$-long 3-path configuration in $G$.
	Then, $\mathcal{P}$ can be extended to a short $t$-model of $\ThetaZero$ such that $|R|\leq 6t+18$.
	Moreover, such a model can be constructed in polynomial time in the size of $G$.
\end{lemma}
\begin{proof}
	Let $s_i$ and $t_i$ be the start and end of $P_i$.
	In the $\Theta_0$ graph, we denote the three parallel edges $e_1,e_2,e_3$ and the vertices as $a,b$.
	
	We set $\eta(a)\coloneqq \{s_1, s_2, s_3\}$ .
	Recall that $\{s_1, s_2, s_3\}$ either induces a triangle or is a single vertex.
	As $t_1,t_2,t_3$ are in the same connected component of $G\setminus \mathrm{Shell}(\mathcal{P})$, 
	we can take an inclusion-wise minimal induced subgraph $T$ of $G\setminus \mathrm{Shell}(\mathcal{P})$
	that is connected and contains $t_1$, $t_2$, and $t_3$. By Lemma~\ref{lem:make-tripod}, $T$ is a tripod with ends $t_1$, $t_2$, $t_3$
	and $T$ can be found in polynomial time.
	Let $s_1', s_2', s_3'$ be the starts of the paths of $T$.
	We set $\eta(b)\coloneqq \{s_1',s_2',s_3'\}$, so we have $|\eta(b)|\le 3$.
	We set $\eta(e_i)$ to contain $s_i$, $P_i$, as well as the induced path in $T$ from $t_i$ to $s_i'$ for all $i\in[3]$.
	By the construction, $\eta(e_i)$ contains an induced path on at least $t+2$ vertices for all $i\in[3]$.
	We set the reserved set $R_0$ to contain only the vertices required by \cref{def:short-model}. 
	Note that for each edge we add at most $2t+6$ guards.
	In particular, the set of guards contains $P_i$ for all $i\in[3]$.
	We observe that by the construction all $e_1,e_2,e_3,a,b$ touch only as allowed in \Cref{def:model}.
	We also observe that $\eta_0(e_i)-R_0(e_i)$ is the shortest path between its endpoints and therefore there is no vertex in its neighborhood that sees more than a $3$-vertex subpath of it.
	Indeed, that would have been a shorter path.
	It remains to consider an ear $S$ of distance at least $t+3$ on at most $t$ vertices.
	If $S$ is an ear attached to a single edge $e_i$ for some $i\in[3]$, it is immediately a shorter path and so a contradiction.
	Hence, consider $S$ between edges $e_i, e_j$ for $i\neq j\in[3]$.
	However, as $R_0$ contains a prefix of $\eta_0(e_i,b)$ of length $t+3$, replacing this prefix with $S$ yields a tree $T$ with fewer edges, a contradiction.
	Therefore, we have obtained a short $t$-model of $\Theta_0$.
\end{proof}

With these lemmata in hand, it is easy to see that in $n^{\Oh(t)}$ time
one can check if a graph contains a $\Xi_{\geq t}$ and, in case of a positive answer,
produce a short $t$-model of $\ThetaZero$.

\begin{lemma}\label{lem:detect-L3PC}
	Given an $n$-vertex graph $G$ and an integer $t$,
	in time $n^{\Oh(t)}$ we can correctly conclude that $G$ does not contain any $\Xi_{\geq t}$
	or find a short $t$-model $\mH = (H,\eta,R)$ in $G$ of $\ThetaZero$ such that $|R|\leq 6t+18$.
\end{lemma}
\begin{proof}
	As every partial $t$-long $3$-path configuration involves $3t+1$ or $3t+3$ vertices,
	in $n^{\Oh(t)}$ time we can enumerate all partial $t$-long $3$-path configurations
	in $G$ and check which of them is promising.

	\cref{lem:xi-to-partial-l3pc} allows us to return a negative answer if no 
	promising partial $t$-long $3$-path configuration is found. 
	Otherwise, any promising partial $t$-long $3$-path configuration
	can be turned in polynomial time into the desired short $t$-model of $\ThetaZero$ by 
	\Cref{lem:extend_tripod_to_theta}.
\end{proof}

\subsection{Initial model}
We now present an initial process that augments a short $t$-model
obtained from \cref{lem:detect-L3PC} to achieve the following property:
deletion of a single edge from it does not destroy all theta minor models. 

This initial process sometimes gets stuck, realizing that it already produced
an object suitable for recursion. This is formalized as a winning scenario:

\begin{definition}[winning scenario]\label{def:winning-scenario}
	Let $C$, $t$ be positive integers and $G$ be a graph.
	A \emph{$C$-winning scenario} for $G$ and $t$ 
	is a set $X \subseteq V(G)$ of size at most $C \cdot t$
	such that one of the following is satisfied: 
	\begin{enumerate}
		\item $G-N[X]$ has no $\Xi_{\geq t}$
		\item there exists a connected component $D$
		of $G-N[X]$ such that both $G[D]$ and $G-N[D]$ contain a $\Xi_{\geq t}$.
	\end{enumerate}
\end{definition}

\begin{lemma}\label{lem:detect-win}
	Given an $n$-vertex graph $G$ and positive integers $C$, $t$, 
	one can in $n^{\Oh(Ct)}$ time find a $C$-winning scenario for $G$ and $t$,
	or correctly conclude that none exists.
\end{lemma}
\begin{proof}
	Within the given running time bound, we can iterate over all candidates for the set $X$.
	For each set $X$ and each connected component $D$ of $G-N[X]$, we use 
	\cref{lem:detect-L3PC} to check if $G[D]$ and/or $G-N[D]$ contain a $\Xi_{\geq t}$.
\end{proof}

We are now ready to present the analog of \cref{lem:split-one-component-from-model}.

\begin{lemma}\label{lem:first-ears}
	There exists a universal constant $C$ such that 
	for every positive integer $t$ and an $n$-vertex graph $G$, if $G$ and $t$ does not admit a $C$-winning scenario,
	then one can in $n^{\Oh(t)}$ time find a $t$-short model $\mH = (H,\eta,R)$ in $G$
	such that $H$ is subcubic of minimum degree at least $2$, $|E(H)| = \Oh(1)$, $|R| = \Oh(t)$
	and for every $e \in E(H)$, $H \setminus \{e\}$ contains $\Theta$ as a minor.
\end{lemma}

\begin{proof}
	The constant $C$ is a sufficiently large constant, stemming from the proof below.
	Instead of providing it explicitly, we use the big-$\Oh$ notation
	and argue that some sets in the proof are of size $\Oh(t)$.
	The final constant is $C$ needs to be big enough
	so that $C\cdot t$ is larger than all the $\Oh(t)$-bounds appearing in the proof.

	As $G$ and $t$ does not admit a $C$-winning scenario, \Cref{lem:detect-L3PC} allows us to find a short $t$-model $\mH=(H, \eta, R)$ of $\Theta_0$.
	If $\mH$ is not proper-maximal, we add any shortest ear between two different edges of $H$. 
	In this case, it is straightforward to check that the produced short $t$-model satisfies the requirements
	of the lemma: after deleting any edge of $H$, the remaining graph still contains $\Theta$ as a minor (see \Cref{fig:theta_with_one_inter_edge_ear}).
	Clearly, $H$ is always cubic, $|E(H)|=4$ and $|R|=\Oh(t)$, as adding a single ear increases the size $R$ by $\Oh(t)$.
	\begin{center}
		\begin{figure}[ht]
			\centering
			\begin{tikzpicture}[scale=0.85]
				\tikzstyle{bluenode}=[draw,circle,fill=black,minimum size=3pt,inner sep=0pt]
				\tikzstyle{texte} =[fill=white, text=black]
				\draw (-1.5,0) node[bluenode] (a) [] {};
				\draw (1.5,0) node[bluenode] (b) [] {};
				\draw (0,0) node[bluenode] (c) [] {};
				\draw (0,0.75) node[bluenode] (d) [] {};
				
				\draw[bend left=20] (a) edge []  node [] {} (d);
				\draw[bend left=20] (d) edge []  node [] {} (b);
				\draw[] (a) edge []  node [] {} (c);
				\draw[] (b) edge []  node [] {} (c);
				\draw[] (d) edge []  node [] {} (c);
				\draw[ bend right=60] (a) edge []  node [] {} (b);
			\end{tikzpicture}
			\caption{A $\Theta_0$ with an ear connecting distinct edges added.}
			\label{fig:theta_with_one_inter_edge_ear}
		\end{figure}
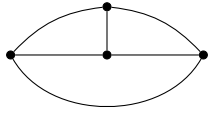
	\end{center}
	
	Suppose then that $\mH$ is proper-maximal, and consider the connected components of $G-N[R]$.
	If no component contains a $\Xi_{\geq t}$,
	then $X \coloneqq R$ is a $C$-winning scenario of the first type for $G$ and $t$. If at least two components contain a $\Xi_{\geq t}$, $X \coloneqq R$ is a $C$-winning scenario of the second type.
	Thus, we may assume that exactly one component, say $D$, of $G-N[R]$ contains a $\Xi_{\geq t}$.
	If $D$ is disjoint from $N[V(\mH)]$, then again by taking $X\coloneqq R$ we satisfy the second variant of a winning scenario,
	as $G[V(\mH)]$ contains a $\Xi_{\geq t}$.
    Otherwise, by \cref{lem:max-model-no-inter-ear}, 
	$D$ is disjoint from $N[\eta(v)]$ for every $v \in V(H)$ and 
	there exists exactly one $e=uv\in E(H)$ such that $N[\eta(e)]$ is intersected by $D$.
	
	A subpath of $2t+7$ vertices of $\eta(e)$ is called \emph{slider} of $\eta(e)$, and we denote the set of sliders as $\sS$.
	For any slider $S$ let us denote its middle vertex as $o_S$. 
	If $\sS$ is empty, then $\eta(e)$ has at most $2t+6$ vertices.
	Since $R$ includes both length $t+3$ prefixes of $\eta(e)$, we would have $\eta(e) \subseteq R$, contradicting $D \cap N[\eta(e)] \neq \emptyset$.
	Hence, $\sS$ is nonempty.
	
	For every $S\in\sS$ we can extend $\mH$ by subdividing the edge $uv$ --- adding a vertex $o_S$ to $H$ and adding $S$ to $R$ --- let us denote this model as $\mH(S)$.
	For any such $\mH(S)$, if we can add an ear joining edges $uo_S$ and $vo_S$, we let $P(S)$ be any shortest possible such ear, and denote the resulting model as $\mH'(S) = \mH(S)+P(S)$; the endpoints of $P(S)$ are called $x_u(S)$ and $x_v(S)$.
	If no ear can be added, we take $P(S)=\varnothing$ and $\mH'(S) = \mH(S)$.
	Note that since $\mH$ is a short $t$-model, and the endpoints of $P(S)$ are at distance at least $t+3$, \Cref{lem:short-model-preserved} ensures that $\mH'(S)$ is also a short $t$-model.
	Denote $\mH'(S) = (H'(S),\eta'(S),R'(S))$.
	Note that $|R'(S)| = \Oh(t)$.
	
	We first deal with the case where there exists $S \in \sS$ such that $P(S) \neq \emptyset$ and $\mH'(S)= \mH(S)+P(S)$
	is not proper-maximal.
	Then, add to it a shortest ear connecting two distinct edges, and let $\mH''(S) = (H''(S),\eta''(S),R''(S))$ be the resulting short $t$-model.
	We claim that $\mH''(S)$ can be returned as the outcome of the lemma. It is clear that $|H''(S)| = \Oh(1)$, and $|R''(S)| = \Oh(t)$.
	It now remains to verify that, after deleting any edge in $H''(S)$, the remaining subgraph still contains a $\Theta$.

	This is a tedious case-by-case check, see \cref{fig:theta_with_two_ears}.
	As $S$ is fixed in the argument of this paragraph, 
	for brevity we drop the argument $S$ from notation like $H'(S)$, $P(S)$, $x_u(S)$, etc.
	Let $f$ be the new edge added to $\mH'$ to obtain $\mH''$.
	It is straightforward to check that $ux_u$ and $vx_v$ are the only two edges of $H'$ whose
	deletion destroys all theta minors in $H'$. Hence, we need only to verify the deletion of these two edges in $H''$
	(or possibly parts of them, if they are subdivided in $H''$).
	As $\mH$ is proper-maximal, $f$ connects two edges among $\{ux_u, x_uo_S, o_Sx_v, x_vv, P\}$. 
	If $f$ connects two edges among $\{x_uo_S, o_Sx_v, P\}$, then $H''$ contains a theta
	disjoint from $ux_u$ and $vx_v$ and we are done.
	By symmetry, let us assume that $f$ has one endpoint in $ux_u$;
	let $y_u$ be the resulting subdivision vertex on $ux_u$ and let $y_v$ be the subdivision vertex at the other end of $f$.
	The deletion of $uy_u$ from $H''$ leaves a theta formed by $f$, $y_ux_u$, $x_uo_S$, $o_Sx_v$, $P$,
	and possibly $x_vy_v$ if the other endpoint of $f$ is on $vx_v$.
	The same theta is left if $vx_v$ or $vy_v$ (in case $y_v$ is on $vx_v$) is deleted.
	The deletion of $y_ux_u$ from $H''$ leaves a path from $u$ to $v$ via $uy_u$, $f$, and $y_v$,
	thus a theta remains with degree-3 vertices $u$ and $v$. 
	A symmetric analysis handles the case when $y_v$ lies on $x_vv$ and $x_vy_v$ is deleted.
	This finishes the proof of the claim that $\mH''(S)$ is a valid outcome of the lemma in the case where $\mH'(S)$ is not proper-maximal.
	
	\begin{figure}[ht]
		\centering
		\begin{tikzpicture}[scale=1.5]
			\tikzstyle{bluenode}=[draw,circle,fill=black,minimum size=4pt,inner sep=0pt]
			\tikzstyle{texte} =[fill=white, text=black]
			\draw (-1.5,0) node[bluenode] (a) [] {};
			\draw (1.5,0) node[bluenode] (b) [] {};
			\draw (0,0) node[bluenode] (c) [] {};
			\draw[left] (a) node {$u$};
			\draw[right] (b) node {$v$};
			\draw[below] (c) node {$o_S$};
			
			\draw (-0.65,0) node[bluenode] (e1) [] {};
			\draw (0.65,0) node[bluenode] (e2) [] {};

			\draw[below] (e1) node {$x_u(S)$};
			\draw[below] (e2) node {$x_v(S)$};
			
			\draw[bend left=70] (a) edge []  node [] {} (b);
			\draw[] (a) edge []  node [] {} (c);
			\draw[] (b) edge []  node [] {} (c);
			\draw[bend right=70] (a) edge []  node [] {} (b);
			
			\draw[bend left=70] (e1) edge [] node [] {} (e2);

			\draw[line width=5pt,red] (a) -- (e1);
			\draw[line width=5pt,red] (b) -- (e2);
		\end{tikzpicture}\\
		\begin{tikzpicture}[scale=1.5]
			\tikzstyle{bluenode}=[draw,circle,fill=black,minimum size=4pt,inner sep=0pt]
			\tikzstyle{texte} =[fill=white, text=black]
			\draw (-1.5,0) node[bluenode] (a) [] {};
			\draw (1.5,0) node[bluenode] (b) [] {};
			\draw (0,0) node[bluenode] (c) [] {};
			
			\draw (-1,0) node[bluenode] (d1) [] {};
			\draw (1,0) node[bluenode] (d2) [] {};
			\draw (-0.65,0) node[bluenode] (e1) [] {};
			\draw (0.65,0) node[bluenode] (e2) [] {};

			\draw[below] (d1) node {$y_u$};
			\draw[below] (d2) node {$y_v$};
			
			\draw[bend left=70] (a) edge []  node [] {} (b);
			\draw[] (a) edge []  node [] {} (c);
			\draw[] (b) edge []  node [] {} (c);
			\draw[bend right=70] (a) edge []  node [] {} (b);
			
			\draw[bend left=70] (d1) edge [] node [] {} (d2);
			\draw[bend left=70] (e1) edge [] node [] {} (e2);

		\end{tikzpicture}
		\hspace{0.3cm}
		\begin{tikzpicture}[scale=1.5]
			\tikzstyle{bluenode}=[draw,circle,fill=black,minimum size=4pt,inner sep=0pt]
			\tikzstyle{texte} =[fill=white, text=black]
			\draw (-1.5,0) node[bluenode] (a) [] {};
			\draw (1.5,0) node[bluenode] (b) [] {};
			\draw (0,0) node[bluenode] (c) [] {};
			
			\draw (-1,0) node[bluenode] (d1) [] {};
			\draw (0,0.4) node[bluenode] (d2) [] {};
			\draw (-0.65,0) node[bluenode] (e1) [] {};
			\draw (0.65,0) node[bluenode] (e2) [] {};

			\draw[below] (d1) node {$y_u$};
			\draw[above right] (d2) node {$y_v$};
			
			\draw[bend left=70] (a) edge []  node [] {} (b);
			\draw[] (a) edge []  node [] {} (c);
			\draw[] (b) edge []  node [] {} (c);
			\draw[bend right=70] (a) edge []  node [] {} (b);
			
			\draw[bend left=70] (d1) edge [] node [] {} (d2);
			\draw[bend left=70] (e1) edge [] node [] {} (e2);

		\end{tikzpicture}
		\caption{Above: $\mH'(S)$ for $P(S) \neq \emptyset$ has two edges whose deletion destroys all thetas --- these are marked red.
		Below: the two most interesting cases how an ear can be added to $\mH'(S)$.}
		\label{fig:theta_with_two_ears}
	\end{figure}
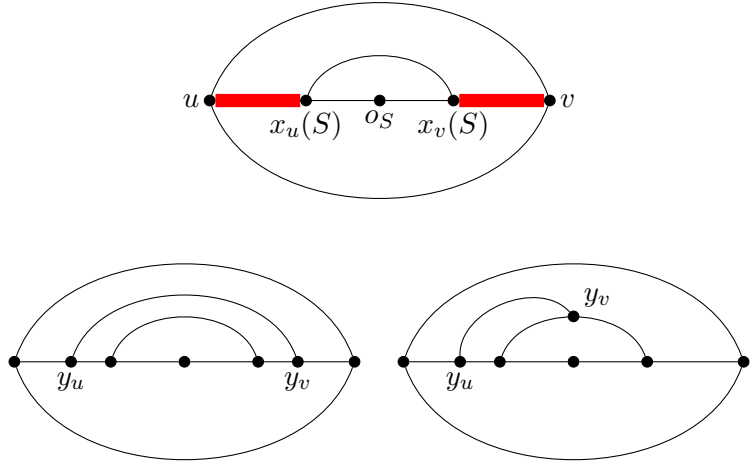

	We may now handle the case where $\mH'(S)$ is proper-maximal for every $S \in \sS$. 
	As $D$ is the only component of $G-N[R]$ containing $\Xi_{\geq t}$ and $R \subseteq R'(S)$,
	every connected component of $G-N[R'(S)]$ containing a $\Xi_{\geq t}$ is in fact
	a connected component of $G[D\setminus N[R'(S)]]$.
	If for some $S \in \sS$ there is no such component, or if there are at least two, then $X\coloneqq R'(S)$ is a $C$-winning scenario of the first or second type respectively.

	Therefore, we may further we assume that for any slider $S$, there exists exactly one component $D(S)$ of $G[D-N[R'(S)]]$
	that contains a $\Xi_{\geq t}$.
	By \cref{lem:max-model-no-inter-ear}, $D(S)$ 
	intersects the neighborhood of at most one edge bag (and no neighborhood of a vertex bag) of the model $\mH'(S)$
	As $\mH$ contains a theta, $G[V(\mH)]$ contains a $\Xi_{\geq t}$;
	if $D(S)$ does not intersect a neighborhood of an edge of $H$, then 
	$G-N[D(S)]$ contains the $\Xi_{\geq t}$ contained in $G[V(\mH)]$ and we are done with the second variant
	of the winning scenario for $X \coloneqq R'(S)$. 
	
	Hence, as $D(S) \subseteq D$ and $e$ is the unique edge of $H$ such that $D$ intersects $N[\eta(e)]$,
	$D(S)$ intersects exactly one of $N[\eta(uo_S)]$ or $N[\eta(vo_S)]$.
	Depending on which one is intersected, we say that $S$ is of $u$-type or $v$-type, respectively.
	(Note that this definition is correct regardless of whether $P(S)$ is empty or not; if $P(S) \neq \emptyset$,
	then $uo_S$ and $vo_S$ are subdivided with $x_u(S)$ and $x_v(S)$, respectively.)
	
	Note that the slider which is closest to $u$ (i.e. minimizing the length of $uo_S$) is of $v$-type, as $uo_S$ is then contained in $R'(S)$.
	Similarly, the slider closest to $v$ is of $u$-type.
	Therefore, we can find a pair of consecutive sliders where the one (denoted $S_u$) closer to $v$ is of $u$-type
	and the one (denoted $S_v$) closer to $u$ is of $v$-type.
	
	Consider now $X \coloneqq R(S_u) \cup R(S_v)$. We claim that $X$ is a $C$-winning scenario; clearly $|X| = \Oh(t)$.
	By contradiction, assume there exists exactly one component $D_X$ of $G-N[X]$ that contains a $\Xi_{\geq t}$. 
	We have $D_X \subseteq D(S_u) \cap D(S_v)$. 
	As $D(S_u)$ intersects $N[V(\mH)]$ only in $N[\eta(uo_{S_u})]$ and $D(S_v)$ intersects
	$N[V(\mH)]$ only in $N[\eta(vo_{S_v})]$, the intersection of $D_X$ with $N[V(\mH)]$
	is contained in $N[\eta(uo_{S_u}) \cap \eta(vo_{S_v})]$. 
	However, as the model is short, $N[\eta(uo_{S_u}) \cap \eta(vo_{S_v})] \subseteq N[R(S_u) \cup R(S_v)]$.
	Hence, $D_X$ is disjoint from $N[V(\mH)]$. As $G[V(\mH)]$ contains
	a $\Xi_{\geq t}$ and this $\Xi_{\geq t}$ is anti-adjacent to $D_X$, we are in the second variant of the winning scenario.
\end{proof}

\subsection{Proof of \cref{thm:EP-3LPC-algo}}
With \cref{lem:first-ears}, we are ready to prove \cref{thm:EP-3LPC-algo}.

\algEPforXis*
\begin{proof}
	\newcommand{\lsimt}{\lambda_{\Theta}}

	Our algorithm follows the same strategy as the one described in~\Cref{thm:EP-cycles-algo}, except for its corner cases which rely on the preliminary results of the current section.
	Let $\lsimt > 0$ be such that $\ftheta(k) \leq \lsimt k \log (2k)$ for every integer $k \geq 1$,
	where $\ftheta$ is the function of~\Cref{thm:thetas-in-cubic}. 
	Let $C$ be the universal constant given by~\Cref{lem:first-ears}, and $c_R,c_E$ be the constants hidden in the $|R| = \Oh(t)$ and $|E(H)| = \Oh(1)$ notation in~\Cref{lem:first-ears}.
	The algorithm returns an integer $k$, a $k$-induced packing $\cXi$, and some $X$ of size bounded by $\lambda k t \log(2k)$, with $\lambda = \max(C, 32 (c_R + 28 \lsimt + 182c_E))$.
	
	We first run~\Cref{lem:detect-L3PC} to check if $G$ contains a $\Xi_{\geq t}$.
	If not, we return $k=0$, an empty \Lpgt{} packing, and $X = \emptyset$.
	Otherwise, we obtain a short $t$-model $\mH_0'$ of $\ThetaZero$.
    Then, we run the algorithm of~\Cref{lem:detect-win} with $C,t$, and distinguish its two cases.

	Assume first that~\Cref{lem:detect-win} outputs a $C$-winning scenario $W$, with $|W| \leq  Ct$.
	If $G-N[W]$ is \Lpgt-free, we return $k=1$, a \Lpgt{} obtained from $\mH_0'$ via \cref{lem:xis-from-model}, and $X=W$, noting $|X| \leq \lambda t$ as desired. 
	Otherwise, let $D$ be the output connected component of $G-N[W]$ and $G' = G - N[D]$, such that both $D$ and $G'$ contain a \Lpgt.
	We recurse on $D$, which returns $k_D,X_D$ and a \Lpgt{} packing $\cXi_D$; and on $G'$, which returns $k',X'$ and $\cXi'$.
	Then, we return
	\[
    k \coloneqq k_D+k', \quad 
    \cXi \coloneqq \cXi' \cup \cXi_D, \quad 
    X \coloneqq W \cup X' \cup X_D.
    \]
	It is easy to see that $\cXi_D \cup \cXi'$ is an induced packing of \Lpgt in $G$ of size $k$, since $G'$ and $D$ are anti-adjacent, and $G-N[X]$ is indeed \Lpgt-free.
	By the definition of a winning scenario, we have $k',k_D \geq 1$, and we distinguish two cases.
	If $k_D \leq k/2$, then:
	\begin{align*}
		|X| &\leq Ct + \lambda t k_D \log(2k_D) + \lambda t k' \log(2k') \\
		&\leq \lambda t \left({C}/{\lambda} + k_D (\log(2k)-1) + k' \log(2k) \right) \tag*{$C \leq \lambda$ and $k_D \geq 1$} \\
		&\leq \lambda t k \log(2k) 
	\end{align*}
	\noindent
	The case is symmetrical when $k' \leq k/2$, yielding the same outcome.
	
	Otherwise, \Cref{lem:detect-win} certifies that $G,t$ does not admit a $C$-winning scenario.
	Then, running~\Cref{lem:first-ears} with $C,t$ outputs a short $t$-model $\mH_0 = (H_0,\eta_0,R_0)$.
	We know that $H_0$ is subcubic, with $|E(H_0)|\leq c_E$, $|V(H_0)| \leq c_E$, $|R_0| \leq c_R t$, and $H_0 -e$ contains a theta for any $e \in E(H_0)$.
	While possible, add a shortest ear between two distinct edges,
	yielding a $t$-ear-decomposition $\mH_0,...,\mH_\ell$, where all steps are proper.
	Take $\mH = (H,\eta,R) = \mH_\ell$, which is a proper-maximal short $t$-model of a subcubic graph, and note that $H -e$ still contains a $\Theta_0$ for any $e \in E(H)$.

	Let $k_\mH^\circ$ be the largest integer such that $n_3(H) \geq \lsimt k_\mH^\circ \log (2k_\mH^\circ) + 8 n_{\leq 2}(H) + 18 n_{||}(H)$,
	or $k_\mH^\circ = 1$ if this integer is nonpositive.
	We then know $H$ contains a packing of $k_\mH^\circ$ disjoint $\ThetaZero$ minors, either by~\Cref{lem:first-ears} or~\Cref{thm:thetas-in-cubic}.
	Now, \Cref{obs:ear-decomp-mostly-cubic} gives the following bounds: $n_{\leq 2}(H) \leq |V(H_0)| \leq c_E$; $n_{||}(H) \leq |E(H_0)| \leq c_E$ since all steps are proper; and $n_3(H) \geq 2 \ell$. Moreover $|R| \leq c_R t + (2t+12) \ell$.
	By the maximality of $k_\mH^\circ$ and the above, we have 
	\[ n_3(H) \leq \lsimt (k_\mH^\circ+1) \log (2k_\mH^\circ+2) + 8 n_{\leq 2}(H) + 18 n_{||}(H). \]
	This gives
	\[ \ell \leq \frac{\lsimt}{2} (k_\mH^\circ+1) \log (2k_\mH^\circ+2) + 13c_E \leq 2 \lsimt k_\mH^\circ \log (2k_\mH^\circ) + 13 c_E.\]
	Therefore, since $\lambda \geq 32 (c_R + 28 \lsimt + 182 c_E)$:
	\begin{align*}
		|R| &\leq c_R t + (2t+12) \ell \leq c_R t + 14 t \ell  \\
		&\leq \left( c_R + 14(2 \lsimt k_\mH^\circ \log(2k_\mH^\circ) + 13c_E) \right)t \\
		&\leq  (c_R + 28 \lsimt + 182 c_E)t k_\mH^\circ \log(2 k_\mH^\circ) \\
		&\leq \frac{1}{32} \lambda t k_\mH^\circ \log (2k_\mH^\circ)
	\end{align*}

	Consider now the connected components $D_1,...,D_p$ of $G-N[R]$ containing a \Lpgt.
	Let $E' \subseteq E(H)$ consist of those edges $e_i$ such that $N[\eta(e_i)] \cap D_i \neq \emptyset$ for some $i$, noting $|E'| \leq p$ by~\Cref{lem:max-model-no-inter-ear}.
	Consider any packing of $k_\mH^\circ$ disjoint $\Theta_0$ minors in $H$ as guaranteed previously.
	In the case where $p=1$, \Cref{lem:first-ears} enables us to always take (at least) one $\Theta_0$ that does not contain $e_1$.
	In all cases, we retain only those $\Theta_0$ that do not contain an edge of $E'$, and project them using~\Cref{lem:xis-from-model} to an induced packing of \Lpgt{}s in $G[V(\mH)]$.
	Then, we let the resulting packing be $\cXi_\mH$, noting it does not touch any $D_i$, $|\cXi_\mH| \geq k_\mH^\circ - p$, and when $p=1$ we have $|\cXi_\mH| \geq 1$.

	For every $i$, we recurse on $G_i = G[D_i]$, obtaining $k_i,\cXi_i,X_i$. 
	We then return:
	\[
    k \coloneqq |\cXi_\mH| + \sum_{i=1}^p |\cXi_i|, \quad 
    \cXi \coloneqq \cXi_\mH \cup \bigcup_{i=1}^p \cXi_i, \quad 
    X \coloneqq R \cup \bigcup_{i=1}^p X_i.
    \]
	It is clear that $\cXi$ is an induced \Lpgt{} packing of size $k$, and that $G-N[X]$ does not contain any \Lpgt, so it remains to bound $|X|$.
	We first deal with the case $p=1$, and for $p \geq 2$ we distinguish two cases according to $|\cXi_\mH|$, similarly to~\Cref{thm:EP-cycles-algo}.
	
	If $p=1$, then we have $k_\mH^\circ \leq |\cXi_\mH| + 1$ and $|\cXi_\mH| \geq 1$, giving:
	\begin{align*}
		|X| = |R| + |X_1| &\leq \frac{\lambda}{32} t k_\mH^\circ \log(2 k_\mH^\circ) + \lambda t k_1 \log(2k_1) \\
		&\leq  \lambda t \log(2k) \left((|\cXi_\mH|+1)/32 + (k-|\cXi_\mH|)\right) \\
		&\leq \lambda t k \log(2k) \tag*{since $|\cXi_\mH| \geq 1$.}
	\end{align*}
	We now proceed to the case $p \geq 2$.
	If $|\cXi_\mH| \geq k_\mH^\circ/2$, then we have $\sum_{i=1}^p k_i \leq k-k_\mH^\circ/2$, giving:
    \begin{align*}
    |X| = |R| + \sum_{i=1}^p |X_i| &\leq \frac{\lambda}{32} t k_\mH^\circ \log(2k_\mH^\circ) + \sum_{i=1}^p \lambda t k_i \log(2k_i)\\
    &\leq t \log(2k) \cdot \left(\frac{\lambda}{32} k_\mH^\circ + \lambda (k - k_\mH^\circ/2)\right) \\
	&\leq \lambda tk \log(2k).
    \end{align*}

	Otherwise, if $|\cXi_\mH| < k_\mH^\circ/2$, in particular $p > k_\mH^\circ/2$ as $|\cXi_\mH| \geq k_\mH^\circ - p$.
    Without loss of generality, assume $k_1 \geq k_2 \geq \ldots \geq k_p$, so $k_i \leq k/i$ for all $i \in [1,p]$, and $\log(2k_i) \leq \log(2k) - \log(i)$.
    \begin{align*}
    |X| &= |R| + \sum_{i=1}^p |X_i| \leq \frac{1}{32}\lambda tk_\mH^\circ \log(2k_\mH^\circ) + \sum_{i=1}^p \lambda t k_i \log(2k_i)\\
    &\leq \frac{1}{16}\lambda tp \log(4p) + \lambda t \sum_{i=1}^p k_i \left(\log(2k) - \log(i)\right)\\
    &\leq \frac{1}{8}\lambda tp \log(2p) + \lambda tk \log(2k) - \lambda t \sum_{i=2}^p k_i \log(i).
    \end{align*}
    As $p \geq 2$, we have
    \[ \sum_{i=2}^p k_i \log(i) \geq \sum_{i=\lceil \frac{p+1}{2} \rceil}^p \log\left(\left\lceil \frac{p+1}{2} \right\rceil\right) \geq \frac{p}{2} \cdot \frac{1}{4}\log(2p) = \frac{1}{8} p \log(2p).\]
    Thus, we have $|X| \leq \lambda tk \log(2k)$ in any case, which concludes the proof.
\end{proof}

	\section{Induced Erd\H{o}s--P\'osa property for long thetas induced minors}\label{sec:L3PC-cor}
	
In this section, we deduce \cref{thm:EP-thetas} from \cref{thm:EP-3LPC-algo}.
Recall that, unlike \cref{sec:cycles}, this result is only existential (not algorithmic).
We are going to work with a short $t$-model of the $\ThetaZero$ graph.

\EPTheta*

We split the whole proof into three logical parts.
First, we introduce some terminology and preliminaries (\cref{sec:big}); then we show how to construct a short $t$-model of $\ThetaZero$ (\cref{sec:initialTheta}); and we conclude with the proof of \cref{thm:EP-thetas} (\cref{sec:thetaProof}).

\subsection{Big component and measure}\label{sec:big}

We start by introducing our notation and terminology.
We would like to abstract the idea of a large component in a graph, while allowing what is large to be measured differently.
We introduce this abstraction as a general concept, but we only use it for a particular measure.
We first specify how we identify the big component. 
In the following, let $G$ be a graph.
We define a \emph{guidance} as an oracle that, for every induced subgraph $G'\subseteq G$, returns a single connected component (called \emph{the big component}) denoted by $\BIG(G')$, or concludes that such a component does not exist (in that case we write $\BIG(G')=\emptyset$).
We require the guidance to be monotone in the following sense:
\[ A\subseteq B\subseteq V(G) \Longrightarrow \BIG(G[A])\subseteq \BIG(G[B]).\]

Note that we will use the notation of $\BIG$ whenever we assume a guidance without explicitly mentioning the connection.
We say that $S \subseteq V(G)$ is a \emph{dominated guided separator} if $\BIG(G-N[S])=\emptyset$.
A simple guidance (which is also easy to compute algorithmically) for vertex-weighted graphs is one that, for every induced subgraph $G'\subseteq G$, returns the connected component of $G'$ of weight greater than half of the total weight of $G$ as $\BIG(G')$ (if such a component exists).
In that case, the dominated guided separator is exactly the dominated balanced separator.
In the proof of \cref{thm:EP-thetas}, we use a guidance that returns the connected component $C$ of $G'$ such that $\pack_{\Theta_t}(C)> \frac{1}{2}\pack_{\Theta_t}(G)$.

Consider a graph $G$ with a guidance.
In our arguments, we will make use of a certain measure of vertex subsets of $G$, indicating how many neighborhoods we need to delete to separate a given set from the big component of the remaining graph:
\begin{definition}[Guided measure]
Let $G$ be a graph with a guidance.
Given $X\subseteq V(G)$, we define the \emph{measure} ($\mu(X)$) as 
\[\mu(X)\coloneqq \min \left\{\left|M \right|~\Big|~M\subseteq V(G) \wedge \BIG(G-N[M])\cap X=\emptyset\right\}. \]
Then a respective \emph{separator} ($\sep(X)$) is defined as a set $M\subseteq V(G)$ such that $\BIG(G-N[M])\cap X=\emptyset$ and $|M|=\mu(X)$.
\end{definition}

We prove basic properties of the guided measure.
For a graph $G$, we use $\mu_{G}\coloneqq \mu(V(G))$.

\begin{observation}[Basic properties of the measure]\label{obs:measure}
  Let $X,Y\subseteq V(G)$ and let $z\not\in X$ be a vertex.
  Then:
  \begin{itemize}
    \item[1.]\phantomsection\label{obs:measure:p2} $\mu(X\cup Y)\le \mu(X)+\mu(Y)$.
    \item[2.]\phantomsection\label{obs:measure:p4} $\mu(N[X])\le |X|$.
    \item[3.]\phantomsection\label{obs:measure:p1} $\mu(X\cup N[z]) \ge \mu(X)\ge \mu(X\cup N[z])-1 $.
    \item[4.]\phantomsection\label{obs:measure:p3} $\mu_{G}= \max_{X\subseteq V(G)} \mu(X)$.
  \end{itemize}
\end{observation}

\begin{proof}
For \hyperref[obs:measure:p2]{Point~1}, observe that $\sep(X)\cup\sep(Y)$ separates $X\cup Y$ from the big component, using the monotonicity of the guidance.
For \hyperref[obs:measure:p4]{Point~2}, we observe that $N[X]$ separates $N[X]$ from the big component.
For \hyperref[obs:measure:p1]{Point~3}, the part claiming $\mu(X\cup N[z]) \ge \mu(X)$ is trivial.
The other part follows by~\hyperref[obs:measure:p2]{Point~1} with $X\coloneqq X$ and $Y\coloneqq N[z]$.
  \hyperref[obs:measure:p3]{Point~4} is a  trivial observation.
\end{proof}

We now define guided analog of the Gyárfás path notion~\cite{Gyarfas}.
\begin{definition}[Guided Gyárfás Path]\label{def:gp}
Let $G$ be a graph with a guidance.
An induced path $P$ of $G$ with a designated start and end vertex
is a \emph{guided Gyárfás path} if $\BIG(G-N[P]) = \emptyset$, but for all proper prefixes $P'$ of $P$, we have $\BIG(G-N[P']) \neq \emptyset$.
\end{definition}

\begin{lemma}\label{lem:Gyarfas}
  Let $G$ be a graph with a guidance. 
There is a guided Gyárfás Path $P$ in $G$.
Moreover, if $\mu(N[P])\le m$ then $G$ has a dominated guided separator of size at most $m$.
\end{lemma}

The proof of \cref{lem:Gyarfas} mimics the original proof of Gyárfás~\cite[Theorem~2.4]{Gyarfas}; the same argument gives the desired outcome in our setting.

\begin{proof}[Proof of \cref{lem:Gyarfas}]
  We first show that there is an induced path $P$ satisfying \cref{def:gp}.
In case $\BIG(G)=\emptyset$ then we take $P=\emptyset$.
Otherwise take $P_1$ as an arbitrary single vertex of $\BIG(G)$.
If $\BIG(G-N[P_1]) = \emptyset$ we set $P\coloneqq P_1$.
Otherwise, for $i\ge1$, suppose that we constructed a path $P_i=p_1\ldots p_i$ and that $D_i\coloneqq\BIG(G-N[P_i])\neq\emptyset$.
We keep an invariant that $p_i$ was chosen to have a neighbor in $D_{i-1}$ (where $D_0\coloneqq \BIG(G-p_1)$).
Since $p_i$ was chosen with a neighbor in $D_{i-1}$, there is a path in $G[D_{i-1}\cup\{p_i\}]$ from $p_i$ to $D_i$ due to the monotonicity of guidance ($D_i\subseteq D_{i-1}$).
  We set $p_{i+1}$ to be the last vertex of this path that belongs to $N(p_i)$.
  Then $p_{i+1}$ has a neighbor in $D_i$ and it is anti-adjacent to $P_{i-1}$.
  Hence $P_{i+1}\coloneqq P_i+p_{i+1}$ is again an induced path.
  Since $G$ is finite, the process stops and outputs a guided Gyárfás path $P$.

  For the second point, we assume that $\mu(N[P])\le m$.
Let $S\coloneqq\sep(N[P])$.
  If $R\coloneqq\BIG(G-N[S])\neq\emptyset$ then $R$ has to intersect $N[P]$.
However, that contradicts the definition of $\mu(N[P])$.
\end{proof}

\subsection{Finding a short $t$-model of $\ThetaZero$}\label{sec:initialTheta}

\begin{lemma}[Finding a short $t$-model of $\ThetaZero$]\label{lem:initial_theta}
  Let $G$ be a connected graph with a guidance, and let $t\ge 1$ be an integer.
  Then one of the following holds:
  \begin{itemize}
    \item $G$ admits an $\Oh(t)$-dominated guided separator, or
    \item $G$ contains a short $t$-model $(\ThetaZero,\eta,R)$ 
      such that $|R|=\Oh(t)$.
\end{itemize}
\end{lemma}

\newcommand{\cA}{4t+1}
\newcommand{\cC}{9t+4}
\newcommand{\cB}{8t+3}
\newcommand{\cD}{3t+1}

\begin{proof}
We take a guided Gyárfás path $P$ given by~\cref{lem:Gyarfas}.
If $\mu(N[P])\leq \cB$, we are done, as there is a $(\le\cB)$-dominated guided separator by \cref{lem:Gyarfas}. 
We assume $\mu_G\ge \mu(N[P])>\cB$.

We take the shortest prefix $P'$ of $P$ such that $\mu(N[P'])=\cB$.
Note that by \cref{obs:measure}, \hyperref[obs:measure:p1]{Point~3}, we know that such a prefix exists
and $P' \neq P$ thus $\BIG(G-N[P'])\neq \emptyset$ by \cref{def:gp}.

Let $Q_1$ and $Q_2$ be the shortest prefix and the shortest suffix of $P'$, respectively, such that \[\mu(N[Q_1])=\mu(N[Q_2])=\cA.\]
By \cref{obs:measure}, \hyperref[obs:measure:p1]{Point~3}, both $Q_1,Q_2$ exist.
Moreover, they are disjoint: otherwise $Q_1\cup Q_2=P'$, and by \cref{obs:measure}, \hyperref[obs:measure:p2]{Point~1}, we would get $\mu(N[P'])\le \mu(N[Q_1])+\mu(N[Q_2])=2 \cdot (\cA)<\cB$, a contradiction.
In particular, as $\mu(N[P'])=\cB$, there is a vertex
\[r \in L\coloneqq N[P'] \setminus N[V(Q_1)\cup V(Q_2)]\setminus P'\] 
that belongs to the big component of $G-N[\sep(N[Q_1])\cup\sep(N[Q_2])]$. 
Indeed, $|\sep(N[Q_1])|=|\sep(N[Q_2])|=\cA<\frac{\cB}{2}$, so there are also vertices of $N[P']-N[V(Q_1)\cup V(Q_2)]$ in the big component
of $G-N[\sep(N[Q_1]) \cup \sep(N[Q_2])]$.
Consider now a shortest path from $r$ to $\BIG(G-N[P'])$ (which is nonempty) in the graph $G- N[(V(Q_1)\cup V(Q_2))]$.
Let $v$ be the last vertex of this path that belongs to $L$. 
It follows that $v$ has a neighbor $u$ in $\BIG(G-N[P'])$.

We view the path $P'$ as starting on the left and ending on the right.
Now, we consider $N(v)\cap P'$ and handle three cases, after which we define vertices $v_\ell,v_r$ and an object $v'$ which will either be a vertex or a triangle.
\begin{enumerate}
  \item\label{case:first}
    If $|N(v)\cap P'|=1$, then we set $v'$ to be the only vertex in $N(v)\cap P'$.
    Then we set $v_\ell$ (resp. $v_r$) to be its left (resp. right) neighbor on $P'$.
  \item\label{case:second}
    If $N(v)\cap P'$ is exactly the edge $ab$ of $P'$.
    Then we set $v'$ as the triangle $abv$. 
    Then we set $v_\ell$ (resp. $v_r$) to be the left neighbor of $a$ (resp. the right neighbor of $b$) on $P'$.
  \item\label{case:third}
    Otherwise, we let $v'\coloneqq v$ and $v_\ell$ (or $v_r$) be the leftmost (or the rightmost) neighbor of $v$ on $P'$, respectively.
    Note that, as $P$ is a guided Gyárfás path, $P$ is an induced path, and we are not in Case~\ref{case:second}, $v_\ell$ and $v_r$ are nonadjacent.
\end{enumerate}
Let $Q'_1$ be $t$ consecutive vertices of $P'$ starting with $v_\ell$ and continuing towards the start of $P'$.
Such vertices must exist as $\cA=\mu(N[Q_1])\le |Q_1|$ (by \cref{obs:measure:p4}~\hyperref[obs:measure:p4]{Point~2}) while $v'\not\in V(Q_1)$.
We perform the analogous construction with $Q'_2$.
Let $Q'_2$ be $t$ consecutive vertices of $P'$ starting with $v_r$ and continuing towards the end of $P'$.
Such vertices must exist as $\cA=\mu(N[Q_2])\le |Q_2|$ (by \cref{obs:measure:p4}~\hyperref[obs:measure:p4]{Point~2}) while $v'\not\in V(Q_2)$.

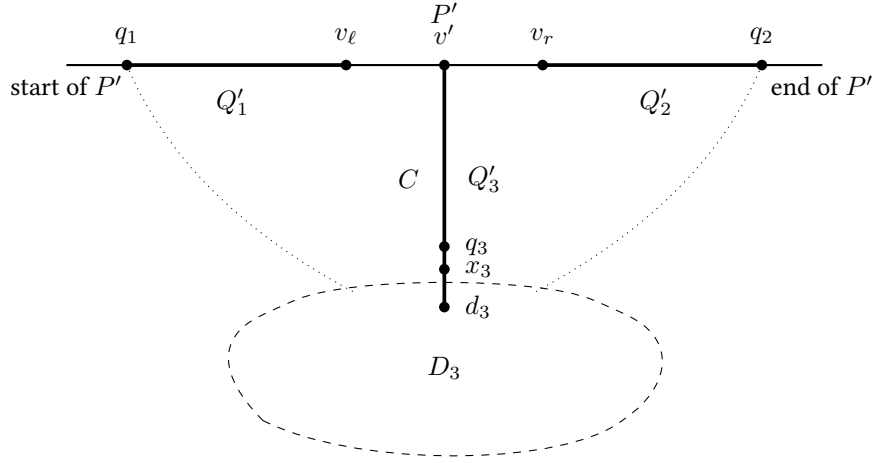
\begin{figure}[t]
\centering
\begin{tikzpicture}[scale=1.0, every node/.style={font=\small}]
  \coordinate (L) at (-5,0);
  \coordinate (R) at (5,0);
  \draw[thick] (L) -- (R);

  \coordinate (vl) at (-1.3,0);
  \coordinate (vm) at (0,0);
  \coordinate (vr) at (1.3,0);

  \fill (vl) circle (2pt) node[above=4pt] {$v_\ell$};
  \fill (vm) circle (2pt) node[above=4pt] {$v'$};
  \fill (vr) circle (2pt) node[above=4pt] {$v_r$};

  \draw[line width=1.4pt] (-4.2,0) -- (vl);
  \draw[line width=1.4pt] (vr) -- (4.2,0);

  \node[below=5pt] at (-2.8,0) {$Q'_1$};
  \node[below=5pt] at (2.8,0) {$Q'_2$};

  \fill (-4.2,0) circle (2pt) node[above=4pt] {$q_1$};
  \fill (4.2,0) circle (2pt) node[above=4pt] {$q_2$};

  \coordinate (newq3) at (0,-2.4);
  \fill (newq3) circle (2pt) node[right=4pt] {$q_3$};
  \coordinate (x3) at (0,-2.7);
  \fill (x3) circle (2pt) node[right=4pt] {$x_3$};
  \coordinate (q3) at (0,-3.2);
  \draw[line width=1.4pt] (vm) -- (q3);
  \fill (q3) circle (2pt) node[right=4pt] {$d_3$};
  \node[right=5pt] at (0,-1.5) {$Q'_3$};

  \node[above=12pt] at (0,0) {$P'$};

  \draw[dashed, rounded corners=8pt]
    (-2.4,-4.7) .. controls (-1.2,-5.3) and (1.2,-5.3) .. (2.4,-4.7)
    .. controls (3.0,-4.1) and (3.0,-3.6) .. (2.2,-3.2)
    .. controls (1.2,-2.8) and (-1.2,-2.8) .. (-2.2,-3.2)
    .. controls (-3.0,-3.6) and (-3.0,-4.1) .. (-2.4,-4.7);

  \node at (0,-4.0) {$D_3$};

  \draw[dotted] (-4.2,0) .. controls (-3.7,-1.2) and (-2.7,-2.2) .. (-1.2,-3.0);
  \draw[dotted] (4.2,0) .. controls (3.7,-1.2) and (2.7,-2.2) .. (1.2,-3.0);
  \draw[dotted] (q3) -- (0,-3.1);

  \node[left] at (-0.2,-1.5) {$C$};

  \node[below] at (-5,0) {start of $P'$};
  \node[below] at (5,0) {end of $P'$};
\end{tikzpicture}
\caption{The situation after \Cref{cl:q3}. The path $P'$ contains two long subpaths
$Q'_1,Q'_2$ starting at $v_\ell,v_r$, while $Q'_3$ starts at $v'$ and avoids
$N[P']\setminus\{v',v\}$. 
Let $q_3$ be the terminal vertex of $Q'_3$
The big component of $D_3\coloneqq \BIG(G-(N[P']\cup N[Q'_3]))\neq \emptyset$ contains a vertex $d_3$ having a neighbour $x_3\in N(q_3)\setminus \left(N[P']\cup N[Q'_3-\{q_3\}]\right)$.
Intuitively, path $d_3x_3q_3$ guarantees sort of an exclusive connection between $D_3$ and $q_3$.
Later in the proof, we show that $D_3$ has similar exclusive connection to a predecessor of $q_1$ on $P'$ and a successor of $q_2$ on $P'$.}
\end{figure}

\begin{claim}\label{cl:q3}
    Suppose $G$ does not admit an $(9t+4)$-dominated guided separator.
Then  there is a path $Q'_3$ of length $t$ starting at $v'$ and avoiding $N[P']\setminus\{v',v\}$.  
  Moreover, if $q_3$ is the terminal vertex of $Q'_3$ and
  \[
    D_3\coloneqq \BIG(G-(N[P']\cup N[Q'_3])),
  \]
  then $D_3\neq \emptyset$ and there are adjacent vertices $x_3,d_3$ such that
     $x_3\in N(q_3)\setminus \left(N[P']\cup N[Q'_3-\{q_3\}]\right)$
    and $d_3\in D_3$.
\end{claim}
\begin{proof}[{Proof of \Cref{cl:q3}}]
  We proceed with the construction by adding single vertices using the Gyárfás path construction principle starting with vertex $v'$ as follows.
  In the first step, we take $u$, the neighbor of $v'$ in the big component of $G-N[P']$, except for Case~\ref{case:first} when $v\in N(v')$ and only $v$ has the neighbor $u$ with the properties above.
In that case, we take both $v$ and $u$ as the first two steps.

  In each subsequent step, if there is a big component in $G-N[P']-N[Q'_3]$, we add a vertex adjacent to the last vertex of the partial path $Q'_3$ constructed so far and remove all other neighbors of the considered last vertex.
We finish the above procedure when $Q'_3$ has length exactly $t$.
Suppose for a contradiction that we were not able to add a new vertex at some point of the construction above. 
This means the big component did not exist.
However, we removed a neighborhood of at most $t+1$ vertices in the construction of $Q'_3$ plus another at most $\cB$ vertices to separate $N[P']$ from the big component.
That means that $\sep(N[P'])\cup Q'_3$ is an $(\cC)$-dominated guided separator, a contradiction.
Moreover, the above argument also applies if $\BIG(G-(N[P']\cup N[Q'_3]))= \emptyset$ at the end of the selection of $Q'_3$.
Thus, we may assume that
\[
  D_3\coloneqq \BIG(G-(N[P']\cup N[Q'_3]))\neq \emptyset.
\]
Let $q_3$ be the terminal vertex of $Q'_3$, and let $\widehat Q_3\coloneqq Q'_3-q_3$.
Just before adding $q_3$, the current big component was
\[
  \widehat D\coloneqq \BIG(G-(N[P']\cup N[\widehat Q_3])).
\]
By construction, $q_3\in \widehat D$.
Moreover, by monotonicity, $D_3\subseteq \widehat D$.
Since $\widehat D$ is connected, consider a shortest path in $G[\widehat D]$ from $q_3$ to $D_3$.
Let $x_3$ be the last vertex of this path that belongs to $N(q_3)$, and let $d_3$ be the vertex immediately after $x_3$ on this path.
Then $x_3$ and $d_3$ are adjacent.
Furthermore, as $x_3\in\widehat D$, we have $x_3\notin N[P']\cup N[\widehat Q_3]$.
The suffix of the path from $d_3$ to $D_3$ avoids $N[P']\cup N[Q'_3]$, and hence $d_3$ belongs to the same connected component of $G-(N[P']\cup N[Q'_3])$ as $D_3$.
\renewcommand{\qedsymbol}{$\diamondsuit$}
\end{proof}

If $G$ admits $(9t+4)$-dominated guided separator, then we are done, 
so assume \cref{cl:q3} gave us $Q_3'$, $D_3$, $q_3$, $d_3$, and $x_3$. 
We take the tripod $\mathcal C\coloneqq v',Q'_1,Q'_2,Q'_3$.
Note that $N[C]$ is not a dominated guided separator of $G$ as (using \cref{obs:measure}~\hyperref[obs:measure:p4]{Point~2}): \[\mu(N[\mathcal C])\le |V(\mathcal{C})|=\cD<\mu_G.\]

We find $\ThetaZero$ by \cref{lem:extend_tripod_to_theta}.
We show that $\mathcal C$ form a promising partial $t$-long $3$-path configuration in $G$ (recall \cref{def:partial_tL3PC}) with $t\coloneqq t$.
We prove that $G-\mathrm{Shell}(\mathcal C)$ has a connected component containing $q_1,q_2,q_3$ (endpoints of $Q'_1,Q'_2,Q'_3$ opposite to $v'$, respectively).
We take the big component of  $G-\mathrm{Shell}(\mathcal C)$.
As $\mu(N[Q_1])=\cA> |V(\mathcal C)|$, $\BIG(G-\mathrm{Shell}(\mathcal C)$ intersects $N[Q_1]$.
We take a vertex in $\BIG(G-\mathrm{Shell}(\mathcal C))\cap N[Q_1]$ closest to $q_1$.
From the intersection we follow $P'$ to reach $q_1$.
An analogous argument and construction hold for $q_2$.
The fact that $\BIG(G-\mathrm{Shell}(\mathcal C))$ contains $q_3$ follows by \Cref{cl:q3} which provides connection to $D_3$ via vertices $x_3$ and $d_3$.
\end{proof}

As a short $t$-model of $\ThetaZero$ contains a $\Xi_{\geq t}$ (\cref{lem:xis-from-model}),
we have the following immediate corollary of \cref{lem:initial_theta}
applied to the guidance that points to the component containing more than half
of the weight of $G$.
\begin{corollary}\label{cor:xi-dbs}
  For every positive integer $t \geq 1$, the class of graphs without
  a $\Xi_{\geq t}$ admits an $\Oh(t)$-dominated balanced separator.
\end{corollary}

\subsection{Proof of \cref{thm:EP-thetas}}\label{sec:thetaProof}

\EPTheta*

\begin{proof}
  For brevity, we shorten $\pack_{\Theta_t}$ to $\pack$.

  We claim that it suffices to prove the statement with an additional assumption
  that $G$ has no $\Xi_{\geq t}$.
  Indeed, let $G$ and $t$ be as in the statement.
  We invoke~\cref{thm:EP-3LPC-algo} for the graph $G$ and integer $t$,
  obtaining $k_\Xi$, a packing $\mathcal{C}_\Xi$ of $k_\Xi$ anti-adjacent copies of $\Xi_{\geq t}$,
  and a set $X_{\Xi}$ of size $\Oh(tk_{\Xi} \log(k_{\Xi}))$ such that $G-N[X_\Xi]$ has no $\Xi_{\geq t}$.
  As every $\Xi_{\geq t}$ contains a $\Theta_t$ induced minor, 
  the subgraph of $G$ induced by the vertices used in $\mathcal{C}_{\Xi}$
  contains $k_{\Xi}\Theta_t$ as an induced minor. Hence, $\pack(G) \geq k_{\Xi}$.

  The graph $G' \coloneqq G-N[X_\Xi]$ has no $\Xi_{\geq t}$, so 
  set $k' \coloneqq \pack(G')$ and assume that there is
  a set $X'$ of size $\Oh(tk' \log(k'))$ so that $G'-N[X']$ has
  no $\Theta_t$ as an induced minor. Then, $G-N[X_\Xi \cup X']$ has
  no $\Theta_t$ as an induced minor and 
  $\pack(G) \geq \max(k_\Xi,k')$.

  Hence, in the remainder of the proof we prove
  the statement of the theorem
  with an additional
  assumption that $G$ has no $\Xi_{\geq t}$.
  We proceed by induction on $\pack(G)$, proving a bound $|X| \leq \lambda kt \log(2k)$ for $k \coloneqq \pack(G)$, 
  where $\lambda > 0$ is the constant hidden in the big-$\Oh$ notation in \cref{lem:initial_theta} in the bound on the size of the dominated guided separator.
  The base case of no $\Theta_t$ as an induced minor in $G$ is trivial with $k = 0$ and $X = \emptyset$.

  For the inductive step, assume $\pack(G) \geq 1$.
  We invoke \cref{lem:initial_theta} on $G$ with a guidance defined to return the connected component $C$ of $G$ such that $\pack(C)> \frac{1}{2}\pack(G)$.
  By \cref{lem:xis-from-model}, $G$ cannot contain a short $t$-model of $\ThetaZero$ as this implies the existence of $\Xi_{\geq t}$ in $G$.
  Hence, $G$ admits a dominated guided separator $S$ of size at most $\lambda t$.

  Let $D_1,\ldots,D_p$ be the connected components of $G-N[S]$ that contain $\Theta_t$ as an induced minor.
  Since $S$ is a dominated guided separator, for every $1 \leq i \leq p$, 
  we have $\pack(G[D_i]) \leq \frac{1}{2}\pack(G) < \pack(G)$. 
  By the inductive hypothesis, for every $1 \leq i \leq p$,
  if we set $k_i \coloneqq \pack(G[D_i])$, 
  then $G[D_i]$ contains a set $X_i$ of size at most $\lambda tk_i \log(2k_i)$ 
  so that $G[D_i] \setminus N[X_i]$ has no $\Theta_t$ as an induced minor. 
  We have $k_i \geq 1$ as $G[D_i]$ contains $\Theta_t$ as an induced minor.
  
  Note that $k \coloneqq \pack(G) \geq \sum_{i=1}^p k_i$, as the components $D_i$ are pairwise anti-adjacent.

  We claim that
  \[ X \coloneqq S \cup \bigcup_{i=1}^p X_i \]
  satisfies the requirements. Clearly, $G-N[X]$ does not contain $\Theta_t$ as an induced minor. It remains
  to bound the size of $X$.

  Recall that for every $1 \leq i \leq p$, we have $k_i = \pack(G[D_i]) \leq \frac{1}{2} \pack(G) \leq k$,
  hence $\log(2k_i) \leq \log(2k)-1$. Thus,
  \begin{align*}
  |X| &= |S| + \sum_{i=1}^p |X_i| \leq \lambda t + \sum_{i=1}^p \lambda tk_i \log(2k_i) \\
  &\leq \lambda t + \lambda t \sum_{i=1}^p k_i \left(\log(2k)-1\right) \\
  &\leq \lambda t + \lambda t \log(2k) \sum_{i=1}^p k_i  - \lambda t \sum_{i=1}^p k_i \\
  &= \lambda tk \log(2k) + \lambda t \left(1 - \log(2k) \left(k - \sum_{i=1}^p k_i\right) - \sum_{i=1}^p k_i \right) \\
  &\leq \lambda tk \log(2k)
  \end{align*}
  The last inequality follows from $k \geq 1$ so $\log(2k) \geq 1$.
\end{proof}

	\section{From induced Erd\H{o}s--P\'osa property to dominated balanced separators}\label{sec:dbs}
	
In this section, we show that \cref{thm:dbs-final}
is a consequence of the results shown so far
and the following simple observation:
an induced \EP property allows us to reduce the problem of
finding balanced separators in $kH$-(induced-minor-)free graphs
to the same problem for $H$(induced-minor-)free graphs.

\begin{lemma}\label{lem:ep-to-dbs}
	Let $\cH$ be a family of graphs.
    Suppose that:
    \begin{enumerate}
        \item $\cH$-free graphs admit $d$-dominated balanced separators,
        \item there exists a function $f : \N \to \N$ such that for every $k \geq 1$,
        every graph $G$ either satisfies $\pack_\cH(G) \geq k$ or it has a set $X \subseteq V(G)$ of size at most $f(k)$ such that $G - N[X]$ is $\cH$-free.
    \end{enumerate}
    Then, for every $k \geq 1$, graphs with $\pack_\cH(G) < k$ admit $(d+f(k))$-dominated balanced separators.

    Furthermore, if the second property is algorithmic, i.e., there is an algorithm that runs in time $T(n,k,t)$ in $n$-vertex graphs and outputs either the desired set $X$ or $k$ pairwise disjoint and anti-adjacent members of $\cH$ as an induced subgraph of $G$,
    then in time $T(n,k,t) + n^{\Oh(d)}$ one can either find a required balanced separator,
    or $k$ pairwise disjoint and anti-adjacent members of $\cH$ as an induced subgraph of $G$.
\end{lemma}
\begin{proof}
    Let $G$ be a graph with $\pack_\cH(G) < k$.
    By the second assumption, there exists $X \subseteq V(G)$ of size at most $f(k)$ such that $G - N[X]$ is $\cH$-free. By the first assumption, there exists $Y \subseteq V(G) \setminus N[X]$ of size at most $d$ such that $N[Y]$ is a balanced separator in $G - N[X]$.
    It is straightforward to observe that $N[X \cup Y]$ is a balanced separator in $G$.
    As $|X \cup Y| \leq d + f(k)$, the proof is complete.

    To obtain the algorithmic version, we first run the algorithm for the second property; this takes time $T(n,k,t)$.
    If it returns the second outcome, we are done.
    Otherwise, it returns a set $X$ of size at most $f(k)$ such that $G - N[X]$ is $\cH$-free.
    Then we exhaustively look for a set $Y$ of size at most $d$ such that $N[Y]$ is a balanced separator in $G - N[X]$; this takes time $n^{\Oh(d)}$.
\end{proof}

\thmdbsfinal*
\begin{proof}
As every $\Xi_{\geq t}$ contains both $\Theta_t$ and $C_t$ as an induced minor,
the last bullet point implies the first and the second one.

Apply \cref{lem:ep-to-dbs} to $\cH$ being the class of all $\Xi_{\geq t}$s.
\cref{cor:xi-dbs}
provides the first requirement --- $\Oh(t)$-dominated balanced separators for graphs that exclude $\Xi_{\geq t}$.
\cref{thm:EP-3LPC-algo} provides the second requirement --- the induced \EP property.
As \cref{thm:EP-3LPC-algo} is algorithmic, we also obtain the algorithmic version of \cref{thm:dbs-final}.
\end{proof}

	\section{Long thetas in blob graphs}\label{sec:blobs}
	In this section we discuss how to derive \cref{thm:thetas-algo-cmso} from \cref{thm:thetas-algos}.

Let $G=(V,E)$ be a graph.
A \emph{blob graph} $G^\circ$ of $G$ is a graph defined as follows:
\begin{align*}
    V(G^\circ) & = \{X \subseteq V(G) ~\mid~ G[X] \text{ is connected}\} \\
    E(G^\circ) & = \{ XY \mid X,Y \in V(G^\circ), X \neq Y, \text{ and $X$ and $Y$ touch each other} \}.
\end{align*}

We say that a graph $H$ has property ($\star$) if, for every graph $G$, the following are equivalent:
    \begin{itemize}
        \item $G$ is $H$-induced-minor-free,
        \item $G^\circ$ is $H$-induced-minor-free.
    \end{itemize}

\begin{lemma}\label{lem:theta-blobclosed}
    For every $t \geq 1$, the graph $\Theta_{t}$ has property ($\star$).
\end{lemma}
\begin{proof}
    We will show that $G$ contains an induced minor model of $\Theta_{t}$ if and only if $G^\circ$ contains an induced minor model of $\Theta_{t}$.
    The forward direction is straightforward, as $G$ is an induced subgraph of $G^\circ$ (where the vertices of $G$ correspond to singleton subsets in $G^\circ$).

    For the reverse direction, let $G=(V,E)$ be a graph such that $G^\circ$ contains an induced minor model of $\Theta_{t}$.
    Denote the degree-3 vertices of $\Theta_{t}$ as $y$ and $z$,
    and the degree-2 vertices of each of the three paths by $a_1,\ldots,a_{t}$, $b_1,\ldots,b_{t}$, and $c_1,\ldots,c_{t}$, respectively, so that $a_1,b_1,c_1$ are neighbors of $y$, and $a_{t},b_{t},c_{t}$ are neighbors of $z$.
    Let $\cX = \{ X_v \}_{v \in V(\Theta_{t})}$ be an induced minor model of $\Theta_{t}$ in $G^\circ$.

    By \cref{lem:minor-models-degree-2}, we can assume that $|X_{a_i}| = |X_{b_i}| = |X_{c_i}| = 1$ for every $i \in [t]$.
    Let $Y=\bigcup X_y \subseteq V$ and $Z=\bigcup X_z \subseteq V$.
    Note that both $Y$ and $Z$ are connected in $G$.
    
    We aim to apply the following claim.

    \begin{claim}[{\cite[Claim 13]{DBLP:conf/stoc/GartlandLPPR21}}]\label{clm:pathinblob}
        Let $P^\circ=X_1,\ldots,X_t$ be an induced path in $G^\circ$ such that $X_1 \not\subseteq X_2$.
        Let $X'_1 \subseteq X_1 \setminus X_2$ and $X'_t \subseteq X_t$ be nonempty sets.
        Let $P = v_1, v_2, \ldots , v_p$ be a shortest path in $G[\bigcup_{j=1}^t X_j]$
        such that $v_1 \in X'_1$ and $v_p \in X'_t$.
        Then $P$ is induced, $p > t$, and $\{v_2, v_3, \ldots , v_{p-1}\} \cap (X'_1 \cup X'_t) = \emptyset$.
    \end{claim}

    Suppose $t \geq 2$.
    Let $X_{a_1}'$ be a component of $X_{a_1} \setminus Y$ that has a neighbor in $Y$ and in $X_{a_2}$.
    To see that such a component exists, note that $Y \cup X_{a_1} \cup X_{a_2}$ is connected, so there is a path from $Y$ to $X_{a_2}$ contained in this set, and sets $Y$ and $X_{a_2}$ do not touch each other.
    Similarly, let $X_{a_t}'$ be a component of $X_{a_t} \setminus Z$ that has a neighbor in $Z$ and in $X_{a_{t-1}}$. 
    If $t=1$, we set $X_{a_1}'$ be a component of $X_{a_1} \setminus (Y \cup Z)$ that has a neighbor in $Y$ and in $Z$; its existence can be shown analogously.

    Thus, $P_a^\circ = Y, X_{a_1}', X_{a_2},\ldots, X_{a_{t-1}}, X_{a_t}', Z$ is an induced path in $G^\circ$.
    (If $t=1$, this path reduces to $P_a^\circ = Y, X_{a_1}', Z$.)
    Clearly $Y \not\subseteq X_{a_1}'$ as, for example, $Y$ touches $X_{b_1}$ but $X_{a_1}'$ does not.
    Thus, by \cref{clm:pathinblob} there is an induced path $P_a$ in $G$ from a vertex of $Y$ to a vertex of $Z$ with at least $t$ internal vertices, and all internal vertices of $P_a$ are in $\bigcup_{i=1}^t X_{a_i}$.
    Let $P_b$ and $P_c$ be defined analogously for the paths corresponding to $b_1,\ldots,b_t$ and $c_1,\ldots,c_t$, respectively.
    Note that the interiors of these three paths are pairwise non-touching.

    Contracting $Y$ and $Z$ to single vertices, and taking all internal vertices of $P_a$, $P_b$, and $P_c$ gives an induced minor model of $\Theta_{\geq t}$ in $G$.
    As $\Theta_{t}$ is an induced minor of $\Theta_{\geq t}$, we conclude that $G$ contains an induced minor model of $\Theta_{t}$.
\end{proof}

The next lemma comes from~\cite[Theorem 13]{swatpaper}. We note that the referenced result is stated for the case where $H$ is a disjoint union of long cycles, but the proof only requires that long cycles have property ($\star$).
Thus, it yields the following more general statement. The idea of the proof comes from~\cite{DBLP:conf/mfcs/PaesaniPR21}.

\begin{lemma}[{Bonnet, Czyżewska, Masa\v{r}\'ik, Pilipczuk, Rzążewski~\cite[Theorem 13]{swatpaper}}]\label{lem:blobclosed}
    If every component of $H$ has property ($\star$), then $H$ has property ($\star$).
\end{lemma}

Combining \cref{lem:theta-blobclosed} with \cref{lem:blobclosed}, we obtain the following corollary.
\begin{corollary}\label{cor:thetas-blobclosed}
    For every $k,t \geq 1$, the graph $k\Theta_{t}$ has property ($\star$).
\end{corollary}

The next result is also shown for classes excluding disjoint unions of long cycles in~\cite{swatpaper},
but the only requirements needed are the ones listed in the statement.
The original idea comes from the work of Gartland, Lokshtanov, Pilipczuk, Pilipczuk, and Rzążewski~\cite{DBLP:conf/stoc/GartlandLPPR21}.

We say that a \textsf{CMSO}$_2$ formula $\psi$ is \emph{hereditary} if:
($i$) if $G \models \psi$, then $G' \models \psi$ for every induced subgraph $G'$ of $G$, and
($ii$) if $G_1 \models \psi$ and $G_2 \models \psi$, then $G_1 + G_2 \models \psi$, where $G_1 + G_2$ is the disjoint union of $G_1$ and $G_2$.
We note that many natural graph properties, like e.g., planarity, bounded degeneracy, or excluding a fixed graph as a minor, can be defined by hereditary \textsf{CMSO}$_2$ formulae.

\begin{theorem}[Bonnet, Czyżewska, Masa\v{r}\'ik, Pilipczuk, Rzążewski~\cite{swatpaper}]\label{thm:alg-blobclosed}
Let $H$ be a graph such that 
\begin{enumerate}
    \item For every $\epsilon >0$ there is a (quasi)polynomial-time algorithm that,
    given an $n$-vertex weighted graph, returns either $H$ as an induced minor, or an independent set of weight at least $(1-\epsilon)$ times optimum.
    \item $H$ has property ($\star$).
\end{enumerate}
Let $r \geq 0$, $\varepsilon \in (0,1)$ be a real, and $\psi$ be a hereditary \textsf{CMSO}$_2$ formula.
There is an algorithm that, given a graph $G$, in (quasi)polynomial time returns one of the following outputs:
\begin{itemize}
\item an induced minor model of $H$ in $G$, or
\item a solution to $(\tw \leq r,\psi)$-\MWIS of size at least $(1-\varepsilon)$ times the optimum one, or,
\item a correct conclusion that no solution to $(\tw \leq r,\psi)$-\MWIS exists.
\end{itemize}
\end{theorem}

Thus, we obtain \cref{thm:thetas-algo-cmso} as a consequence of \cref{thm:thetas-algos}, \cref{cor:thetas-blobclosed}, and \cref{thm:alg-blobclosed}.

\thmcmso*

	\section{Conclusion and further work}\label{sec:outro}
	Let us conclude the paper by pointing out some open questions and directions for future research.

First of all, let us recall three conjectures concerning graphs that exclude a planar graph as an induced minor, which were the motivation for our work.

\conjEP*

\conjdbs*

\conjalg*

We are very far from confirming these conjectures.
Perhaps an interesting next case would be an analogue of a long $\Theta_{\geq t}$, but with more than three long paths? 

\medskip

Second, recall that in all our induced \EP results, the bound on the size of the hitting set is $\mathcal{O}(tk \log k)$, where $k$ is the number of objects to be packed and $t$ corresponds to the size of a single object.
On the other hand, the lower bound obtained by a modification of a construction in~\cite{DBLP:conf/soda/AhnGHK25} is $\Omega(tk + k \log k)$. We believe it would be interesting to close this gap.
Let us remark that Ahn, Gollin, Huynh, and Kwon~\cite{DBLP:conf/soda/AhnGHK25} conjectured that in their setting, i.e., for induced packing of long, but not necessarily induced cycles, the lower bound is tight, i.e., that the upper bound can be improved to $\mathcal{O}(tk + k \log k)$.

\medskip

Another possible direction is to consider induced \EP property just for \emph{some} cycles.
Given a graph $G$ and a set $S \subseteq V(G)$, an \emph{$S$-cycle} is a cycle that intersects $S$.
As announced by Ahn and Kwon~\cite{EP-S-cycles}, the techniques used to prove \cref{thm:coarseEP} can be extended to show the induced \EP property for long (but not necessarily induced) $S$-cycles.
Our approach does not seem to generalize to such a setting.
Of course, in an analogous way one can define $S$-thetas etc.

\medskip
Finally, let us recall that we proved \emph{algorithmic} induced \EP property for long cycles and for $\Xi_{\geq t}$, but only an existential statement for $\Theta_{\geq t}$.
A fundamental reason for this is that we do not know how to efficiently find $\Theta_{\geq t}$ as an induced minor.
Recently, Dallard, Dumas, Hilaire, Milani\v{c}, Perez, and Trotignon~\cite{DBLP:conf/iwoca/DallardDHMPT24} showed that this problem is polynomial-time solvable for $t=1$, but the general case remains open.

It appears that the difficulty lies in detecting minor models that resemble the one in \cref{fig:turtle}, i.e., the image of each of degree-3 vertices is a path.
A naive approach would be to guess the images of degree-2 vertices, and then try to connect the endpoints by finding two disjoint and anti-adjacent paths between specified terminals.
However, the latter problem is \NP-hard even in restricted classes of graphs~\cite{DBLP:conf/icalp/AboulkerBPT25}.
Thus, a polynomial-time algorithm for detecting $\Theta_{\geq t}$ as an induced minor would require a more involved approach.
We also emphasize that there are some constant-size graphs $H$ (even trees) such that it is \NP-hard to detect $H$ as an induced minor~\cite{DBLP:conf/icalp/AboulkerBPT25,DBLP:conf/soda/KorhonenL24}.

	\bibliographystyle{alphaurl}
	\bibliography{bibliography}
	
	\onlysoda{%
		\appendix
		\section{Models: postponed proofs}\label{sec:models-later}
		\lemmodeladdear*
\begin{proof}
\begin{enumerate}[nosep, label=(\roman*)]
    \item In both cases, $\eta'(x)$ and $\eta'(y)$ are connected, they do not touch, and do not touch any previous vertex bag, since those are contained in $R$.
    \item Since $P$ is of distance at least $t+3$ on $\mH$, 
    \begin{enumerate}[nosep]
      \item in the case $e=f$, $\eta'(g)$ is of length at least $t+3$, and 
      \item in both cases, because $\mH$ is short, $\eta'(g')$ is of length more than $t$. 
    \end{enumerate}
    \item All other new edge bags contain paths of length at least $t+3$, 
    because $R$ contains the length $t+3$ prefixes of $\eta(e)$ and $\eta(f)$.
    \item The fact that new edge bags touch $\eta'(x),\eta'(y)$ only through (one of) their endpoints is by definition. The relation to preserved vertex bags is inherited from $\eta(e)$ and $\eta(f)$.
    \item All incidences between new edges are sound by construction, $\eta'(g')$ and $\eta'(g)$ do not touch other edge bags, and remaining new edge bags inherit their incidences from $\eta(e)$ and $\eta(f)$.
\end{enumerate}
\end{proof}

\leminitcycle*

\lemshortmodelpreserved*
}
\end{document}